\documentclass[review,onefignum,onetabnum]{siamart171218}

\usepackage[normalem]{ulem}
\usepackage{xcolor}
\usepackage{enumerate}

\usepackage{lipsum}
\usepackage{amsfonts}
\usepackage{graphicx}
\usepackage{epstopdf}
\usepackage{algorithmic}
\usepackage{comment}
\ifpdf
  \DeclareGraphicsExtensions{.eps,.pdf,.png,.jpg}
\else
  \DeclareGraphicsExtensions{.eps}
\fi

\newsiamremark{remark}{Remark}
\newsiamremark{hypothesis}{Hypothesis}
\crefname{hypothesis}{Hypothesis}{Hypotheses}
\newsiamthm{claim}{Claim}

\headers{Sensitivity analysis of chaotic flows using the adjoint shadowing stabilized march}{P. Thakur and S. Nadarajah}

\title{A stabilized march approach to adjoint-based sensitivity analysis of chaotic flows \thanks{
\funding{This work was funded by the Fonds de recherche du Québec – Nature et technologies (FRQNT) Doctoral Award, McGill Engineering Doctoral Award (MEDA) and the Natural Sciences and Engineering Research Council of Canada (NSERC) Discovery Grant RGPIN-2019-04791.}}}

\author{Pranshul Thakur \thanks{Department of Mechanical Engineering, McGill University, Montreal, H3A 0C3, QC, Canada (\email{pranshul.thakur@mail.mcgill.ca}).}
\and Siva Nadarajah \thanks{Department of Mechanical Engineering, McGill University, Montreal, H3A 0C3, QC, Canada (\email{siva.nadarajah@mcgill.ca}).} }

\usepackage{amsopn}
\usepackage{cancel}

\DeclareMathOperator{\diag}{diag}

\makeatletter
\newcommand*{\addFileDependency}[1]{
  \typeout{(#1)}
  \@addtofilelist{#1}
  \IfFileExists{#1}{}{\typeout{No file #1.}}
}
\makeatother

\ifpdf
\hypersetup{
  pdftitle={A stabilized march approach to adjoint-based sensitivity analysis of chaotic flows},
  pdfauthor={P. Thakur, S. Nadarajah}
}
\fi

\begin{document}

\maketitle

\begin{abstract}
Adjoint-based sensitivity analysis is of interest in computational science due to its ability to compute sensitivities at a lower cost with respect to several design parameters. However, conventional sensitivity analysis methods fail in the presence of chaotic flows. Popular approaches to chaotic sensitivity analysis of flows involve the use of the shadowing trajectory. The state-of-the-art approach computes the shadowing trajectory by solving a least squares minimization problem, resulting in a space-time linear system of equations. The current paper computes the adjoint shadowing trajectory using the stabilized march, by specifying the adjoint boundary conditions instead of solving a minimization problem. This approach results in a space-time linear system that can be solved through a single backward substitution of order $\mathcal{O}(n_u^2)$ with $n_u$ being the dimension of the unstable subspace. 
It is proven to compute sensitivities that converge to the true sensitivity for large integration times and that the error in the sensitivity due to the discretization is of the order of the local truncation error of the scheme. The approach is numerically verified on the Lorentz 63 and Kuramoto-Sivasinsky equations.        
\end{abstract}

\begin{keywords}
  chaos, sensitivity analysis, linear response, adjoint, least squares shadowing, stabilized march
\end{keywords}

\begin{AMS}
34A34, 37A99, 37D20, 37D45, 37N30, 46N40, 65P99, 76F20
\end{AMS}

\section{Introduction}
\label{sec:introduction}
Sensitivity analysis of quantities of interest is useful for optimization \cite{nocedal_wright}, computing improved functional estimates \cite{becker_rannacher_2001, giles_suli, pierce_giles}, mesh adaptation \cite{corrigan_et_al_2019,zahr1,thakur_nadarajah_jcp, fidkowski_darmofal_2011, inriagoaloriented, dolejsigoal1} etc. The adjoint method of computing sensitivities is efficient when there are a lot fewer functionals than independent parameters, as is common in aerodynamic applications \cite{jameson_adjoint_optimization}. However, both the conventional forward and adjoint methods yield unbounded sensitivities when applied to chaotic systems \cite{lea_e_al_2000}. The reason is due to the positive Lyapunov exponents in chaotic systems, which cause most nearby trajectories to diverge exponentially.

The work of Ruelle \cite{ruelle_diff_of_srb} established the differentiability of chaotic hyperbolic systems. The ensemble adjoint approach of Lea et al. \cite{lea_ensemble_adjoint, lea_e_al_2000} computes chaotic sensitivities by computing conventional adjoint solutions over short trajectories and taking the ensemble average of the resulting short-window sensitivities. Some approaches replace the initial condition with periodic boundary conditions on the tangent/adjoint equations to yield bounded sensitivities, such as the unstable periodic orbits \cite{lasagna_upo} and periodic shadowing \cite{lasagna_periodic_shadowing}. The unstable periodic orbit approach has also been extended to discontinuous periodic orbits to yield regularized linear responses to parameter perturbations \cite{hicken_chaos}. It is known that hyperbolic systems with a compact attractor have trajectories that shadow the reference trajectory when the function generating the dynamics is perturbed \cite{anosov_shadowing, bowen_shadowing,numerical_shadowing,chater_wang_2017}. Using this idea, Wang et al. \cite{wang_2013,wang_lss} developed a method of linearizing the governing equation using the shadowing trajectory. In the original least squares shadowing (LSS) approach \cite{wang_lss}, a least squares minimization problem is solved to approximate the shadowing trajectory. To reduce the cost of the LSS approach, there have been approaches that reduce the size of the linear system through multiple shooting \cite{blonigan_mss}, reduced-order modeling \cite{vermier_lss, fidkowski_lss} or formulating the shadowing algorithm in a Fourier space with truncation \cite{ashley_hicken_fourier,kantarakias_papadakis}. The non-intrusive \cite{nilss} and adjoint non-intrusive methods \cite{blonigan_adjoint_nilss,nilsas} have been developed which require only the unstable subspace of the dynamical system for forward sensitivity and an additional one-dimensional neutral subspace for adjoint-based sensitivity \cite{nilsas} methods. Their cost scales with the number of unstable Lyapunov covariant vectors in the system and are efficient for problems with lower dimensional unstable subspaces, but are expensive when the number of positive Lyapunov exponents is large \cite{nilsas}. We also note that the shadowing-based approaches might compute shadowing trajectories that are non-physical \cite{chandramoorthy_wang_nonphysical} and approaches based on the linear response \cite{ruelle_diff_of_srb} have been investigated \cite{chandramoorthy_efficient_linear_response,ni2023recursive,sliwiak2023} to address the issue. Since the shadowing-based approaches are known to yield accurate sensitivities for several test cases involving turbulent fluid flows \cite{blonigan_airfoil,nilsas,nilss}, the current paper does not aim to address the issue of non-physical shadow trajectories and assumes that the shadow trajectories are representative of the perturbed dynamical system. Moreover, few recent approaches based on the linear response require the computation of the shadow trajectory \cite{ni_fast,ni2023recursive}, and the approach developed in this work can be used to that end.

The current paper provides a method to compute the adjoint shadowing trajectory through the stabilized march approach \cite{ascher_bvp}. The method uses QR-decompositions to compute the unstable subspace, similar to the non-intrusive adjoint LSS approach (NILSAS) \cite{nilsas}. However, instead of solving a least squares minimization problem in the unstable subspace, the current approach chooses the adjoint boundary conditions to make the resulting space-time linear system triangular. We note that an approach similar to the stabilized march has been used for chaotic maps \cite{ni_fast} by modifying adjoint boundary conditions. The current work can be seen as the extension of stabilized march to chaotic continuous-in-time flows. Additionally, while the classical stabilized march requires the exact number of unstable modes for stability \cite{ascher_bvp}, the method developed in this work is stable even when the number of unstable modes is overestimated at the beginning of the algorithm. The paper is organized as follows: \Cref{sec:notations} discusses the preliminaries and notations used in this work. \Cref{sec:sensitivity_shadowed_trajectory} briefly provides the results for the adjoint of LSS, proven in \cite{thakur_nadarajah_lss}.  \Cref{sec:adjoint_bc_derivation} provides the adjoint boundary conditions used in this work and proves that the sensitivity from the method converges to the true sensitivity at $\mathcal{O}\left(\frac{1}{\sqrt{T}} \right)$ as the integration time $T\to\infty$.  \Cref{sec:stabilized_march} shows that the boundary conditions presented in \cref{sec:adjoint_bc_derivation} convert the linear space-time system to a sequence of triangular solves. \Cref{sec:m_greaterthan_nu} provides the stabilized march algorithm when the dimension of the unstable subspace has been overestimated. \Cref{sec:err_due_to_discretization} investigates errors due to discretization and proves that the discretization error in the computed sensitivity from stabilized march is of the order of the truncation error of the scheme. Finally, \cref{sec:test_cases} applies the approach to the test cases involving chaotic flows.   

\section{Preliminaries}
 \label{sec:notations}
Consider the following dynamical system, governing the evolution of a state $u$ with time $t$:
\begin{equation} 
\label{eq:governing_eq}
    \frac{du}{dt} = f(u(t,s),s),\;\;u(0) = u_0
\end{equation} 
 where $u(t,s) \in \Lambda$ $\forall t$, $\Lambda \subset X$ is a compact attractor in a $n$-dimensional Hilbert space $X$, $s \in \mathbb{R}$ is the design or control parameter (for example, the shape of an airfoil \cite{jameson_adjoint_optimization} or uncertainty parameters to determine robust flight trajectory of a re-entry vehicle \cite{optimal_control}) and $f(u,s)$ is a nonlinear function governing the system. The solution trajectory is used to compute an output, $J(u,s) : \Lambda \times \mathbb{R} \rightarrow \mathbb{R}$. The quantity of interest is a time-averaged output: 

\begin{equation}
\label{eq:time_averaged_functional}
    \bar{J}(s) = \lim_{T\to \infty} \frac{1}{T} \int_0^T J(u,s) dt.
\end{equation}
We assume that the system is ergodic \cite{katok_dynamical_system}. Therefore, $\bar{J}$ does not depend on the initial condition $u_0$ and is only a function of $s$. We also assume that the first and high-order derivatives of $f$ and $J$ appearing in the text exist and are continuous on $\Lambda$.

Perturbing $s$ modifies the nonlinear function $f$ governing the system. Thus, the resulting perturbed trajectory $\tilde{u}$ is different from $u$ and, for chaotic dynamical systems, $\tilde{u}$ diverges exponentially from $u$ in the vicinity of $u$ for most initial conditions $\tilde{u}(0)$. This is due to the positive Lyapunov exponents of chaotic systems and leads to an unbounded growth of the norm of $v = du/ds$ after linearization \cite{lea_e_al_2000}. On the other hand, the true sensitivity
\begin{equation}
    \label{eq:true_sensitivity}
    \frac{d\bar{J}}{ds} = \lim_{\delta s \to 0} \frac{\bar{J}(s+\delta s) - \bar{J}(s)}{\delta s}.
\end{equation}
exists and is bounded for a uniformly hyperbolic dynamical system \cite{ruelle_diff_of_srb}.

Throughout the paper, we assume that the dynamical system \cref{eq:governing_eq} satisfies the assumption of uniform hyperbolicity \cite{katok_dynamical_system}. This means that the tangent space at each point of the attractor $\Lambda$ can be split into invariant stable, unstable, and neutral subspaces. Further, there exists a set of Lyapunov exponents, $\{ \lambda_j\}_{j=1}^n$, governing the rate at which an infinitesimal perturbation in the state grows or shrinks with time. For such systems, the sensitivity exists and is bounded \cite{ruelle_diff_of_srb}. 

For each Lyapunov exponent $\lambda_j$, there is an associated covariant Lyapunov vector (CLV), $\phi_j(u)$, satisfying \cite{ginelli_2007,ruelle_ergodic,eckmann_ergodic, pilyugin_shadowing,wang_2013}:
\begin{equation}
\label{eq:forward_clv}
    \frac{d\phi_j}{dt} = f_u\phi_j - \lambda_j \phi_j.
\end{equation}
Thus, a perturbation in the state along $\phi_j$ grows exponentially at the rate $e^{\lambda_j t}$. The vectors $\{\phi_j(u)\}_{j=1}^n$ are bounded at all times and form a basis of tangent space at each point on the attractor \cite{wang_2013}. Similarly, whenever CLVs exist, Oseledet's multiplicative ergodic theorem \cite{oseledets} also implies the existence of adjoint covariant Lyapunov vectors, $\{\hat{\phi}_j\}_{j=1}^n$, as shown in \cite{adjoint_clv}. The adjoint covariant vectors satisfy \cite{wang_2013, ni_adjoint_arxiv,ginelli_2007}:
\begin{equation}
\label{eq:adjoint_clv}
    -\frac{d\hat{\phi}_j}{dt} = f_u^*\hat{\phi}_j - \lambda_j \hat{\phi}_j.
\end{equation}
Like covariant Lyapunov vectors, the adjoint covariant Lyapunov vectors are always bounded and form a basis of the vector space. Further, by choosing a proper normalization, the Lyapunov covariant and adjoint covariant vectors satisfy \cite{wang_2013}:
\begin{equation}
\label{eq:dot_product_lyapunov_vectors} 
    \phi_i(t)^T\hat{\phi}_j(t) = \begin{cases}
        1, & i=j \\
        0, &  i\neq j
    \end{cases}.
\end{equation}
By the assumption of uniform hyperbolicity, the neutral subspace is one-dimensional \cite{katok_dynamical_system}. Denoting $i=n_0$ to represent the neutral subspace, $\lambda_{n_0} = 0$ and it can be seen from \cref{eq:forward_clv} and the chain rule that $\phi_{n_0} = f$, \cite{wang_2013,blonigan_multigrid}.

Any vector $w$ can be split into $w=w^+ + w^0 + w^-$, where we denote $w^+$, $w^-$ and $w^0$ as its components in the unstable, stable and neutral adjoint subspaces respectively \cite{katok_dynamical_system}. In this work, we denote by $n_u$ the dimension of the unstable subspace. Unless otherwise specified, the norm $\|\cdot\|$ is taken to be the $L_2$-norm defined on the finite-dimensional Hilbert space $X$.

\section{Sensitivity of the functional using the adjoint LSS}
\label{sec:sensitivity_shadowed_trajectory}
The sensitivity of a functional for a chaotic flow can be computed using the method of least squares shadowing \cite{wang_lss}. The adjoint of the least squares shadowing equation, $\psi(t)$, solves the boundary value problem
\begin{subequations}
\label{eq:lss_adjoint}
\begin{align}
    & \frac{d \psi}{dt} + f_u^* \psi + J_u = r,  \label{eq:lss_adjoint_a}\\
    & \frac{dr}{dt} - f_u r = \frac{1}{\alpha^2} \left(\psi^T f + J - \bar{J} \right) f,\label{eq:lss_adjoint_b}
\end{align}
\end{subequations}
with boundary conditions $\psi(0) = 0,\;\psi(T) = 0.$
The sensitivity can then be computed as:
\begin{equation}
\label{eq:adjoint_lss4}
    \frac{d \bar{J}}{ds} = \lim_{T\to\infty} \frac{1}{T} \int_0^T \left(\psi^Tf_s + J_s \right) dt 
\end{equation}
The derivation of \cref{eq:adjoint_lss4} can be found in \cite{wang_lss,thakur_nadarajah_lss}. It was shown in \cite{thakur_nadarajah_lss} that the weak solution to the adjoint LSS equation, \cref{eq:lss_adjoint}, exists and is unique. The following result from \cite[Proposition 6.14]{thakur_nadarajah_lss} pertains to the convergence of the adjoint LSS when the specified adjoint boundary conditions are inhomogeneous but bounded.
\begin{theorem}[{\cite[Proposition 6.14]{thakur_nadarajah_lss}}]
\label{thm:adjoint_lss_theorem}
     For a uniformly hyperbolic dynamical system, consider the finite-time adjoint solutions $\psi_T(t)$ obtained on $t\in[0\;T]$ by solving \cref{eq:lss_adjoint}, with the boundary conditions $\psi_T(0)$ and $\psi_T(T)$ satisfying $\|\psi_T(0)\|<\delta$ and $ \|\psi_T(T)\|<\delta$ $\forall T>0$, where $0<\delta<\infty$ is a constant independent of $T$. Then,
    \begin{equation*}
        \bigg| \frac{d\bar{J}}{ds}\bigg|_{\text{true}} -\lim_{T\to\infty} \frac{1}{T} \int_0^T \left(\psi_T^Tf_s + J_s \right) dt \bigg| \to 0  
    \end{equation*}
    at $\mathcal{O}\left(\frac{1}{\sqrt{T}}\right)$.
\end{theorem}
\section{Finite-time adjoint shadowing direction}
\label{sec:adjoint_bc_derivation}
The conventional adjoint equation for sensitivity analysis,
\begin{equation}
\label{eq:conventional_adjoint2}
    \frac{d\psi}{dt} + f_u^*\psi + J_u = 0,
\end{equation}
with the terminal condition $\psi(T) = 0$, yields an adjoint solution $\psi(t)$ that becomes unbounded backward in time for chaotic flows \cite{lea_e_al_2000}. It was shown in \cite{ni_adjoint_arxiv} that there exists a unique adjoint solution, $\psi^\infty(t)$, known as the adjoint shadowing direction that satisfies \cref{eq:conventional_adjoint2} and also satisfies the following conditions: 
\begin{enumerate}[(i)]
    \item $\psi^\infty(t)$ is bounded on $t \in [0\;\infty)$.
    \item $\psi^{\infty^+}(0) = 0$.
    \item $\frac{1}{T}\int_0^T\psi^{\infty T}fdt = 0$
\end{enumerate}
The adjoint shadowing direction, $\psi^\infty$, yields the correct sensitivity computed using \cref{eq:adjoint_lss4}. A method to approximate $\psi^\infty$ was constructed in~\cite{nilsas} by solving a least squares minimization problem in the unstable subspace. The current approach aims to construct the sequence of finite-time adjoint solutions, $\psi_T(t)$, bounded on $[0\;T]$, which approximates $\psi^\infty(t)$ and yields the correct sensitivity as $T\to\infty$. This approach can be considered as a special case of adjoint LSS, for which it is known that the correct sensitivity is recovered even when the boundary conditions of the adjoint equation are nonhomogeneous \cite{thakur_nadarajah_lss}. In what follows, we specify the boundary conditions that yield the correct sensitivity as $T\to\infty$ for a uniformly hyperbolic dynamical system. The next result shows that the condition (iii) of $\psi^\infty$ can be replaced by a condition at a particular time $t^* \in [0,T]$.
\begin{lemma}
\label{lm:adjoint_particular_time}
    Let the adjoint solution, $\psi(t)$, satisfy \cref{eq:conventional_adjoint2}. Then, the following conditions are equivalent:
    \begin{enumerate}
            \item $\psi^Tf(t^*) = \bar{J}-J(t^*)$ at a particular time $t^* \in [0\;\;T]$.
            
        \item $\psi^Tf(t) = \bar{J}-J(t) \quad \forall t\in[0\;\;T]$.
        
        \item $\frac{1}{T}\int_0^T\psi^Tf dt = 0$.
        
    \end{enumerate}
\end{lemma}
\begin{proof}
    $(1)\implies (2)$: If $(1)$ is true at a particular time $t^*$, taking the inner product of \cref{eq:conventional_adjoint2} with $f(t)$ yields
    \begin{equation*}
        \int_{t^*}^t f^T\left(\frac{d\psi}{dt} + f_u^* \psi + J_u \right)dt =   \int_{t^*}^t \frac{d (\psi^{ T} f + J)}{dt} dt = 0.
    \end{equation*}
    \begin{equation*}
        \implies \psi^T(t)f(t) +J(t) = \psi^T(t^*)f(t^*) + J(t^*) = \bar{J} 
    \end{equation*}
    which shows that $(2)$ holds for any $t\in[0\;T]$.

    $(2)\implies (3)$: Follows from taking the time average on either side of $(2)$.

    $(3) \implies (1)$: It was shown in \cite{thakur_nadarajah_lss} that $(3)\implies(2)$. Moreover, it is clear that $(2)\implies(1)$.
\end{proof}
   \Cref{lm:adjoint_particular_time} shows that if one of the conditions on $\psi(t)$ is true, the remaining two conditions hold.

\begin{theorem}
\label{thm:stabilized_march_1}
    For each $T>0$, associate an adjoint solution $\psi_T(t)$ defined on $t\in[0\;T]$ that satisfies
    \begin{enumerate}[(i)]
        \item $\frac{d\psi_T}{dt} + f_u^*\psi_T + J_u = 0$ on $t \in [0\;T]$.
        \item $\psi_T(T)^Tf(T) = \bar{J}-J(T)$.
        \item $\sup\limits_{T>0} \|\psi_T(0)\| < \delta$ and $\sup\limits_{T>0} \|\psi_T(T)\| < \delta$, where $0<\delta<\infty$ is a constant independent of $T$. 
    \end{enumerate}
    Then,
    \begin{equation*}
        \bigg| \frac{d\bar{J}}{ds}\bigg|_{\text{true}} -\lim_{T\to\infty} \frac{1}{T} \int_0^T \left(\psi_T^Tf_s + J_s \right) dt \bigg| \to 0  
    \end{equation*}
    at $\mathcal{O}\left(\frac{1}{\sqrt{T}}\right)$.
\end{theorem}
\begin{proof}
    For a given $T>0$, conditions $(i)$ and $(ii)$ imply, by \cref{lm:adjoint_particular_time}, that $\psi_T(t)^Tf(t) + J(t) - \bar{J} = 0$ $\forall t \in [0\;T]$. This $\psi_T(t)$ is therefore a special case of the LSS adjoint equation, \cref{eq:lss_adjoint}, with $r = 0$. Since $\psi_T(t)$ satisfies the LSS adjoint equation for all $T>0$ and also satisfies condition $(iii)$, \cref{thm:adjoint_lss_theorem} shows that the sensitivity computed using $\psi_T(t)$ converges to the true sensitivity at $\mathcal{O}\left(\frac{1}{\sqrt{T}}\right)$.    
\end{proof}

\begin{theorem}
\label{thm:h1_norm_error}
    Let $\psi_T(t)$ satisfy the hypothesis of \cref{thm:stabilized_march_1}. Then, the error $e_T(t) = \psi_T(t) - \psi^\infty(t)$, defined on $t\in[0\;T]$ satisfies
    \begin{equation*}
        \|e_T\|^2_{H^1} = \int_0^T e_T^Te_T dt + \int_0^T \frac{de_T}{dt}^T \frac{de_T}{dt} dt < M < \infty
    \end{equation*}
    for any $T>0$, where $M$ is a real positive constant. In particular, $\lim_{T\to\infty} \|e_T\|_{H^1}$ is bounded.
\end{theorem}
\begin{proof}
    Since $\psi_T(t)$ in \cref{thm:stabilized_march_1} is a special case of the LSS adjoint equation, the result follows from \cite[Corollary 6.8]{thakur_nadarajah_lss}.
\end{proof}

\begin{theorem}
   \label{thm:stabilized_march_2a_T}
    Let $n_{u}$ be the dimension of the unstable subspace. For $T>0$, let $h_T$ be a vector satisfying $h_T^Tf(T) = \bar{J} - J(T)$, $\sup\limits_{T>0} \|h_T\| \leq c_h < \infty$, and let $Q_T^+$ be the orthonormal matrix whose columns span the unstable adjoint subspace at time $t=T$. Then, $\exists a_T \in \mathbb{R}^{n_u}$ such that $\psi_T(t)$, with the terminal condition $\psi_T(T) = Q_T^+a_T + h_T$, satisfies the hypothesis of \cref{thm:stabilized_march_1}. Hence, the sensitivity computed using $\psi_T$ converges to the correct sensitivity at $\mathcal{O}\left(\frac{1}{\sqrt{T}}\right)$.   
\end{theorem}

\begin{proof}
    We note that since the columns of $Q_T^+$ span the unstable adjoint subspace, $f(T)^TQ_T^+ = 0$ by \cref{eq:dot_product_lyapunov_vectors}. Hence, $\psi_T(T)^Tf(T) = h_T^Tf(T)= \bar{J} - J(T)$ and the condition $(ii)$ in \cref{thm:stabilized_march_1} is satisfied. By choosing $\psi_T(t)$ that evolves according to the adjoint equation with the given terminal condition, condition $(i)$ in \cref{thm:stabilized_march_1} holds. Hence, it remains to specify $a_T$ to satisfy condition $(iii)$ in \cref{thm:stabilized_march_1}. We note that the elements of the unstable adjoint subspace grow backward in time while those of the stable adjoint subspace grow forward in time \cite{wang_2013}. Since the subspaces are invariant \cite{katok_dynamical_system,ni_adjoint_arxiv}, the adjoint equation can be decoupled into the unstable, stable, and neutral components as 
    \begin{subequations}
        \begin{align}
             & \frac{d\psi^+_T}{dt} + f_u^*\psi^+_T + J_u^+ = 0,\qquad \psi_T^+(T) = Q_T^+a_T + h_T^+ \label{eq:adjoint_split_unstable}\\
             & \frac{d\psi^-_T}{dt} + f_u^*\psi^-_T + J_u^- = 0,\qquad \psi_T^-(T) = h_T^- \label{eq:adjoint_split_stable}\\
            & \frac{d\psi^0_T}{dt} + f_u^*\psi^0_T + J_u^0 = 0,\qquad \psi_T^0(T) = h_T^0 \label{eq:adjoint_split_neutral},
        \end{align}
    \end{subequations}
where each component of the adjoint solution evolves separately due to invariance, and the adjoint solution at any time $t$ is $\psi_T(t) = \psi_T(t)^+ + \psi_T(t)^0 + \psi_T(t)^-$. 

We have the following bound on $\psi_T(T)$:
\begin{equation}
\label{eq:normboundary_T}
    \|\psi_T(T)\| = \|Q_T^+a_T + h_T\| \leq \|a_T\| + \|h_T\|\qquad \forall T>0.
\end{equation}
On the other hand, having evolved the adjoint equation until $t=0$, we obtain the solution $\psi_T(0)$. Since the adjoint space and the unstable adjoint subspace are vector spaces \cite{katok_dynamical_system}, $\psi_T(0)$ can be expressed as the sum of two vectors, one in the unstable adjoint subspace and the other orthogonal to the unstable subspace. Specifically, let $Q_0^+$ be the orthonormal matrix, the columns of which span the unstable adjoint subspace at $t=0$. We then have $\psi_T(0) = Q_0^+a_0 + v_0$, with $a_0 = Q_0^+\psi_T(0)$ and $v_0 = \left(I - Q_0^+Q_0^{+T}\right) \psi_T(0)$. Since $\psi_T(0)^+$ is in the unstable adjoint subspace and hence, in the span of the columns of $Q_0^+$, we have $Q_0^+Q_0^{+T}\psi_T(0)^+ = \psi_T(0)^+$. Therefore, we have the following equality:
\begin{equation}
    v_0= \left(I - Q_0^+Q_0^{+T}\right) \psi_T(0) = \left(I - Q_0^+Q_0^{+T}\right) \left(\psi_T(0)^0 + \psi_T(0)^-\right),
\end{equation}
yielding the following bound on $\psi_T(0)$:
\begin{equation}
\label{eq:normboundary_0}
    \|\psi_T(0)\| \leq \|a_0\| + \|v_0\| \leq \|a_0\| + \|\psi_T^-(0)\| + \|\psi_T^0(0))\|, \qquad \forall T>0,
\end{equation}
since $\|I - Q_0^+Q_0^{+T}\|= 1$. From \cref{eq:normboundary_T,eq:normboundary_0} and using the fact that $\sup\limits_{T>0}\|h_T\|\leq c_h < \infty$, if we show that $\exists a_T$ such that $\|a_T\|$, $\|a_0\|$, $\|\psi_T^-(0)\|$ and $\|\psi_T^0(0))\|$ are bounded by a constant independent of $T$, condition $(iii)$ in \cref{thm:stabilized_march_1} is satisfied and the current theorem is proved. The rest of the proof shows the boundedness of these quantities.

From the general solution to the ODE \ref{eq:adjoint_split_stable} for the adjoint solution at $t=0$, given a known adjoint solution at $t=T$, we have 
\begin{equation}
\label{eq:psi_minus_bounded}
     \begin{aligned}
        \|\psi_T^-(0)\| & = \bigg\|e^{-\int_T^0f_u^*d\tau}\psi^-_T(T) + \int_0^T e^{-\int_{\tau}^0f_u^*d\xi}J_u^-(\tau) d\tau\bigg\|\\   
        &\leq \|e^{-\int_T^0f_u^*d\tau}h_T^-\| + \int_0^T \|e^{-\int_{\tau}^0f_u^*d\xi}J_u^-(\tau) \|d\tau\\
        & \leq \lambda^T \|h_T^-\| + \|J_u\|_\infty \int_0^T \lambda^\tau d\tau \\
        & \leq \sup_{T>0} \|h_T\| + \frac{\|J_u\|_\infty}{|\ln\lambda|} = c_1
          \end{aligned}
\end{equation}
for $\lambda \in (0,1)$ \cite{katok_dynamical_system}. Since $J_u$ is continuous on the compact attractor $\Lambda$, $\|J_u\|_\infty < \infty$. Therefore, $c_1<\infty$.

For the neutral adjoint equation, we note that $J_u^0(t) = \hat{J}_u^0(t) \hat{\phi}_{n_0}(t)$, where $\hat{J}_u^0(t)$ is a scalar. Since $\hat{\phi}_{n_0}(t)$ satisfies $\frac{d\hat{\phi}_{n_0}}{dt} = -f_u^*\hat{\phi}_{n_0}$ (Eq. \cref{eq:adjoint_clv}), we have $\hat{\phi}_{n_0}(t_f) = e^{-\int_{t_i}^{t_f} f_u^*dt}\hat{\phi}_{n_0}(t_i)$ and therefore, the following is obtained using the general solution to the ODE \cref{eq:adjoint_split_neutral}:
\begin{equation}
\label{eq:psi_0_bounded}
    \begin{aligned}
        \psi_T^0(0) & = e^{-\int_T^0f_u^*d\tau}\psi^0_T(T) + \int_0^T e^{-\int_{\tau}^0f_u^*d\xi}J_u^0(\tau) d\tau\\
        &= e^{-\int_T^0f_u^*d\tau}h^0_T + \int_0^T \hat{J}_u^0(t) e^{-\int_{\tau}^0f_u^*d\xi}\hat{\phi}_{n_0}(\tau) d\tau \\
        &= e^{-\int_T^0f_u^*d\tau}h^0_T + \left(\int_0^T \hat{J}_u^0(\tau)d\tau\right)\hat{\phi}_{n_0}(0) \\
        &= e^{-\int_T^0f_u^*d\tau}h^0_T + \left(J(T)-J(0)\right)\hat{\phi}_{n_0}(0), \\
    \end{aligned}
\end{equation}
where we have used the fact that $f(t)^TJ_u(t) = \hat{J}_u^0(t)$ from \cref{eq:dot_product_lyapunov_vectors} and $\frac{dJ}{dt} = J_u\frac{du}{dt} = J_u(t)f(t)$ to yield $\int_0^T \hat{J}_u^0(t)dt = J(T)-J(0)$.
Since $J(u)$ and $\|\hat{\phi}(u)\|$ are continuous real-valued functions on the compact attractor, they are bounded by constants independent of $T$. Moreover, $\|e^{-\int_T^0f_u^*d\tau}h^0_T\| \leq C \|h_T^0\| \leq C \sup\limits_{T>0} \|h_T\|$. Hence, we obtain $$\sup\limits_{T>0} \|\psi^0_T(0)\| \leq C  \sup\limits_{T>0} \|h_T\| + 2|J(t)|_\infty \|\hat{\phi}_{n_0}(t)\|_\infty=c_2 < \infty. $$ 

To obtain bounds on $a_T$, we first express $a_T$ as a function of $a_0$. To that end, we get the solution to the unstable adjoint equation, \cref{eq:adjoint_split_unstable}, at $t=0$ to be
    \begin{equation}
    \label{eq:adjoint_plus}
    \begin{aligned}
             \psi_T^+(0) & = e^{-\int_T^0f_u^*d\tau}\psi^+_T(T) + \int_0^T e^{-\int_{\tau}^0f_u^*d\xi}J_u^+(\tau) d\tau\\
             &= e^{-\int_T^0f_u^*dt} \left( \psi_T^+(T) + \int_0^T e^{-\int_\tau^T f_u^*d\xi}J_u^+(\tau)d\tau \right) \\
             &= e^{-\int_T^0f_u^*dt} \left( Q_T^+ a_T + h_T^+ + \int_0^T e^{-\int_\tau^T f_u^*d\xi}J_u^+(\tau)d\tau \right)
    \end{aligned}
    \end{equation}

We note that 
\begin{equation}
\label{eq:a0_intermediate}
     a_0 = Q_0^{+T}\psi_T(0) =  Q_0^{+T}\psi_T^+(0) +  Q_0^{+T}\psi_T^0(0)+ Q_0^{+T}\psi_T^-(0)
\end{equation}
Left-multiplying \cref{eq:a0_intermediate} by $Q_0^+$ and noting that $Q_0^+Q_0^{+T}\psi_T(0)^+ = \psi_T(0)^+$, one obtains
\begin{equation} 
\label{eq:a_T_a_0_relation1}
     Q_0^+\left(a_0 - Q_0^{+T}\psi_T^0(0)- Q_0^{+T}\psi_T^-(0) \right) = \psi_T^+(0), 
\end{equation}
Substituting \cref{eq:adjoint_plus} on the right hand side of \cref{eq:a_T_a_0_relation1} and solving for $a_T$ yields
\begin{equation}
\label{eq:a_T_a_0_relation2}
\begin{aligned}
      a_T = Q_T^{+T} \Biggl( &e^{-\int_0^Tf_u^*dt}\left( Q_0^+\left(a_0 - Q_0^{+T}\psi_T^0(0)- Q_0^{+T}\psi_T^-(0) \right) \right)  \\
      & - h_T^+ - \int_0^T e^{-\int_\tau^T f_u^*d\xi}J_u^+(\tau)d\tau \Biggr).  
\end{aligned}
\end{equation}
\Cref{eq:a_T_a_0_relation2} expresses $a_T$ as a function of $a_0$. Hence, instead of specifying $a_T$, one can instead specify $a_0$ and obtain $a_T$ from \cref{eq:a_T_a_0_relation2}. We choose $a_0$ to be any bounded vector in $\mathbb{R}^{n_u}$.  
The first term on the right-hand side of \cref{eq:a_T_a_0_relation2} is an application of the adjoint propagation operator in forward time to an unstable subspace. Hence, the term is bounded by 
\begin{equation*}
\begin{aligned}
    \|e^{-\int_0^Tf_u^*dt}\left( Q_0^+\left(a_0 - Q_0^{+T}\psi_T^0(0)- Q_0^{+T}\psi_T^-(0) \right) \right) \| & \leq \lambda^T \left( \|a_0\| + \|\psi_T^0(0)\| + \|\psi_T^-(0)\| \right)\\ &\leq \|a_0\| + c_1 + c_2
\end{aligned}
\end{equation*}

 Since $\|e^{-\int_\tau^T f_u^*d\xi}J_u^+(\tau)\| \leq \lambda^{T-\tau}\|J_u^+(\tau)\| \leq \lambda^{T-\tau}\|J_u\|_\infty$ for $\lambda \in (0,1)$, the third term in \cref{eq:a_T_a_0_relation2} is bounded by 
\begin{equation*}
    \bigg\|\int_0^Te^{-\int_\tau^T f_u^*d\xi}J_u^+(\tau)d\tau\bigg\|  \leq \int_0^T \|e^{-\int_\tau^T f_u^*d\xi}J_u^+(\tau)\| d\tau \leq \frac{\|J_u\|_\infty}{|\ln \lambda|}\qquad \forall T>0.
\end{equation*}
Therefore, since $h_T$ is bounded by $c_h$, $a_T$ from \cref{eq:a_T_a_0_relation2} is bounded by
\begin{equation*}
    \|a_T\| \leq  \|a_0\| + c_1 + c_2 + c_h + \frac{\|J_u\|_\infty}{|\ln \lambda|} = c_3 < \infty
\end{equation*}
 Hence, $\|a_T\|$ is bounded for all $T>0$. Having shown the boundedness of all terms appearing on the right-hand side of the inequalities \cref{eq:normboundary_T,eq:normboundary_0}, the result on convergence of the sensitivity follows from \cref{thm:stabilized_march_1}.
\end{proof}

\cref{thm:stabilized_march_2a_T} shows that if we have a vector $h_T$
that satisfies $h_T^Tf(T) = \bar{J}-J(T)$ and $\sup\limits_{T>0}\|h_T\| < \infty$, then $a_T$ given by \cref{eq:a_T_a_0_relation2} yields an adjoint sensitivity that converges to the true sensitivity. Although there are many possibilities for the choice of $h_T$, the next result provides an expression of a particular choice that is easy to compute for a numerical scheme.

\begin{theorem}
\label{thm:h_T_choice}
    For a uniformly hyperbolic dynamical system, $$h_T = \left( \bar{J}-J(T) \right)\frac{f(T)}{\|f(T)\|^2}$$ satisfies the requirements for $h_T$ in \cref{thm:stabilized_march_2a_T}.
\end{theorem}
\begin{proof}
    It can be seen that $h_T^Tf(T) = \bar{J}-J(T)$. The boundedness of $h_T$ is evident by noting that, for any $u\in\Lambda$, we have from \cref{eq:dot_product_lyapunov_vectors},
    \begin{equation*}
        1 = | f(u)^T \hat{\phi}_{n_0}(u) | \leq \|f(u)\| \|\hat{\phi}_{n_0}(u)\| \leq \tilde{C} \|f(u)\| 
    \end{equation*}
    \begin{equation*}
      \implies  \frac{1}{\|f(u)\|} \leq \tilde{C}, \quad u\in\Lambda
    \end{equation*}

Since for any $T>0$, we have $\|h_T\| = \frac{|\bar{J}-J(T)|}{\|f(T)\|} \leq 2\tilde{C} |J(u)|_\infty$, this shows that $\sup\limits_{T>0}\|h_T\| \leq 2\tilde{C}|J(u)|_\infty$ is bounded for a continuous functional $J$ on a compact attractor. 
\end{proof}

\section{The stabilized march approach}
\label{sec:stabilized_march}
The boundary conditions stated in \cref{thm:stabilized_march_2a_T} can be used to solve for the adjoint solution using the stabilized march approach \cite{ascher_bvp}, which we now present. For the ease of derivations that follow, we keep $T$ fixed and drop the subscript $T$ from the adjoint solution $\psi_T$. \cref{thm:stabilized_march_2a_T} and \cref{thm:h_T_choice} show that one can obtain the correct sensitivity as $T\to\infty$ by solving the following adjoint equation subject to the boundary conditions:
\begin{equation}
\label{eq:adjoint_bcs}
\begin{aligned}
      &\frac{d\psi}{dt} + f_u^*\psi + J_u = 0 \quad t\in[0\;\;T] \\
      &\psi(T) = Q_T^+a_T + \left( \bar{J}-J(T) \right)\frac{f(T)}{\|f(T)\|^2} \\
      &Q_0^{+T}\psi(0) = a_0
\end{aligned}
\end{equation}
where $a_T$ is unknown and needs to be solved, whereas $a_0$ is a specified bounded vector, the choice of which is presented later. To prevent the unbounded growth of roundoff errors, we use the stabilized march approach \cite[Section 4.4.3]{ascher_bvp}, which divides the time interval $[0\;\;T]$ into $K$ segments: $0 = t_0 <t_1<...<t_{K-1}<t_K=T$. On each segment $[t_{i-1},t_i]$, $i=K,..,1$, we can express the solution using the principle of superposition as
\begin{equation}
\label{eq:adjoint_superposition}
    \psi(t) = Y_i(t)a_i + v_i(t), \quad t\in[t_{i-1},t_i]\quad i=K,..,1
\end{equation}
where $Y_i(t) \in \mathbb{R}^{n\times n_u}$ is a fundamental homogeneous solution satisfying the homogeneous adjoint equation
\begin{equation}
\label{eq:hom_adjoint}
    \frac{dY_i(t)}{dt} + f_u^*Y_i(t) = 0, \quad Y_i(t_i) = Q_i, \quad t\in[t_{i-1},t_i], \quad i=K,..,1,
\end{equation}
and $v_i(t) \in \mathbb{R}^n$ is a particular solution satisfying
\begin{equation}
\label{eq:nonhom_adjoint}
  \frac{dv_i(t)}{dt} + f_u^*v_i(t) +J_u = 0, \quad v_i(t_i) = \gamma_i, \quad t\in[t_{i-1},t_i], \quad i=K,..,1.
\end{equation}
At $t=t_K=T$, we choose $Y_K(t_K) = Q_T^+$ and $v_K(t_K) = \left( \bar{J}-J(T) \right)\frac{f(T)}{\|f(T)\|^2}$. It can be seen from \cref{eq:adjoint_superposition} that these choices satisfy the boundary conditions on $\psi(t)$ presented in \cref{eq:adjoint_bcs}. At $t=t_{i-1}$, we have two representations of the adjoint solution $\psi$:
\begin{equation}
\label{eq:adjoint_continuity}
    \psi(t_{i-1}) = Y_i(t_{i-1})a_i + v_i(t_{i-1}) = Y_{i-1}(t_{i-1})a_{i-1} + v_{i-1}(t_{i-1}).
\end{equation}
The terminal conditions $Y_{i-1}(t_{i-1})$ and $v_{i-1}(t_{i-1})$ are chosen to ensure continuity of the solution. This can be done as follows: we compute the QR decomposition $Y_i(t_{i-1}) = Q_{i-1}R_{i-1}$ and set $Y_{i-1}(t_{i-1}) = Q_{i-1}$ to ensure that the $\text{span}\{Y_i(t_{i-1})\} = \text{span}\{Y_{i-1}(t_{i-1})\}$. The terminal value of $v_{i-1}$ is chosen to be
\begin{equation*}
    \gamma_{i-1} = \left(I-Q_{i-1}Q_{i-1}^T\right) v_i(t_{i-1}),
\end{equation*}
which makes $\gamma_{i-1}$ orthogonal to $Q_{i-1}$. We note that a similar procedure was applied in the non-intrusive least squares shadowing approaches \cite{nilss,nilsas,blonigan_adjoint_nilss} before solving the least squares minimization problem in the unstable subspace. 

Left multiplying $Q_{i-1}^T$ in the second equality in the continuity condition, \cref{eq:adjoint_continuity}, yields a relation between $a_i$ and $a_{i-1}$:
\begin{equation}
\label{eq:eqn_in_R}
    R_{i-1}a_i = b_{i-1} + a_{i-1}\quad i=1,..,K;
\end{equation}
where $b_{i-1} = -Q_{i-1}^Tv_i(t_{i-1})$ and
$a_0$ can be chosen to be any fixed vector. In the current work, we set it to $a_0=0$. The unknowns $a_1,a_2,\dots,a_K$ can then be solved sequentially from $i=1$ through $K$ using \cref{eq:eqn_in_R}.
Since each $R_{i-1}$ is upper triangular, the linear system in \cref{eq:eqn_in_R} can be solved directly by backward substitution. The implementation is presented in \cref{alg:m_equals_nu} for the case where the number of homogeneous solutions is exactly equal to the dimension of the unstable subspace. We note that a similar approach of marching has been used in a modified NILSAS for chaotic maps \cite{ni_fast} to solve the linear space-time system efficiently. The current work, on the other hand, accounts for the neutral adjoint subspace and implements marching for the case of flows (continuous-in-time dynamical systems), which have applications in various fields.

We note that \cref{alg:m_equals_nu} would diverge if the number of homogeneous adjoint solutions is greater than the dimension of the unstable subspace. Since the dimension of the unstable subspace is not known in advance, it is preferable to have a method that works when we overpredict the number of unstable modes. The next section presents a way to apply the algorithm for the case $Y_i(t) \in \mathbb{R}^{n\times m}$, where $m>n_u$. 

\begin{algorithm}
\caption{Adjoint shadowing march when $m=n_u$}
\label{alg:m_equals_nu}
\begin{algorithmic}[1]
\STATE{Set integration time $T$ and the spin-up times $T_{0i}$ and $T_{0f}$.

\textbf{Primal solve:}}
\STATE{Choose a random initial condition $u_{-1}$ and solve the primal equation from $t=-T_{0i}$ to $t=0$ to get $u_0$ on the attractor.}
\STATE{Solve the primal equation from $t=0$ to $t=T+T_{0f}$ and save the primal solution at all time steps or at discrete checkpoints of time. 

\textbf{Adjoint solve:}}
\STATE{Choose a random matrix $W\in\mathbb{R}^{n\times m}$ such that the set of its columns and $f(T+T_{0f})$ is linearly independent. Compute $QR$ decomposition of the augmented matrix $[f(T+T_{0f}),\;W]$ and remove the first column of $Q$ to get $\tilde{Q}$.}
\STATE{Solve the homogeneous adjoint equation \cref{eq:hom_adjoint} from $t=T+T_{0f}$ to $t=T$ with the terminal condition $\tilde{Q}$ and factorizations at segments of length $\Delta T$. Save the $Q$ of the final homogeneous adjoint at $t=T$ as $Q_K$. }
\STATE{
Set $\gamma_K = \left(\bar{J}-J\right)\frac{f(T)}{\|f(T)\|^2}$.
}
\FOR{$i=K,K-1,..,1$}
\STATE{Set $Y_i(t_i) = Q_i$. Integrate \cref{eq:hom_adjoint} backward in time until $t=t_{i-1}$. Compute the QR decomposition $Y_i(t_{i-1}) = Q_{i-1}R_{i-1}$ and save $Q_{i-1}$ and $R_{i-1}$.}
\STATE{Set $v_i(t_i) = \gamma_i$ and integrate \cref{eq:nonhom_adjoint} backward in time until $t=t_{i-1}$.}
\STATE{Save $b_{i-1} = -Q_{i-1}^Tv_i(t_{i-1})$.}
\STATE{Set $\gamma_{i-1} = v_i(t_{i-1}) + Q_{i-1}b_{i-1}$.}
\STATE{Save $d_i = \int_{t_{i-1}}^{t_i} Y_i(t)^Tf_s(t)dt$ and $h_i = \int_{t_{i-1}}^{t_i} v_i(t)^Tf_s(t)dt $.}

\ENDFOR

\textbf{Forward march:}

\STATE{Set $a_0=0$.}
\FOR{$i=1,2,..,K$}
\STATE{Solve the upper triangular linear system, $R_{i-1}a_{i} = a_{i-1}+b_{i-1}$, for $a_i$ by backward substitution.}

\ENDFOR

\textbf{Compute sensitivity:}

\STATE{
Sensitivity $\frac{d\bar{J}}{ds} = \frac{1}{T}\left(\sum\limits_{i=1}^K \left(a_i^Td_i + h_i\right) + \int_0^T J_s dt\right)$
}
\RETURN $\frac{d\bar{J}}{ds}$
\end{algorithmic}
\end{algorithm}

\subsection{Adjoint stabilized march when $\mathbf{m>n_u}$}
\label{sec:m_greaterthan_nu}
When we perform the $QR$ decomposition of homogeneous adjoint solutions $Y_i$, after a sufficient number of time steps, the first $n_u$ columns of $Q$ span the unstable adjoint subspace \cite{benettin2}. On the other hand, the upper-triangular matrices $R_i \in \mathbb{R}^{m\times m}$ can be used to compute the Lyapunov exponents of the subspaces, from the highest to the lowest exponents. In particular, one can compute the finite-time Lyapunov exponents as \cite{benettin2}
\begin{equation}
\label{eq:compute_lyapunov_exponents}
    \lambda_j = \frac{1}{T} \sum_{i=1}^K \ln(|[R_{i}]_{jj}|) \quad j=1,\dots,m
\end{equation}
where $[R_i]_{jj}$ is the entry in the $j^{\text{th}}$ diagonal of $R_i$. Having computed $\lambda_1\geq \lambda_2\geq \dots\geq \lambda_m$, we choose $j^*$ such that $\lambda_{j^*}>0$ and $\lambda_{j^*-1} \leq 0$. This gives us the dimension of the unstable subspace as $n_u = j^*$.

The above fact leads to the next result, which shows that \cref{eq:eqn_in_R} can be split into a convenient form when $m>n_u$.

\begin{theorem}
\label{thm:split_stabilized_march}
    For each $i=1,\dots,K$, let $Q_i \in \mathbb{R}^{n \times m}$ be the orthonormal matrix specified in \cref{alg:m_equals_nu}. Let $Q_i = [Q_i^u,\;\;Q_i^{s^0}]$, where $Q_i^u \in \mathbb{R}^{n\times n_u}$ contains the first $n_u$ columns of $Q_i$ and $Q_i^{s^0}$ contains the remaining columns. Similarly, let $a_i = [a_i^u,\;\;a_i^{s^0}]^T$, where $a_i^u \in \mathbb{R}^{n_u}$ and $a_i^{s^0} \in \mathbb{R}^{m-n_u}$. If $a_i$ satisfies \cref{eq:eqn_in_R}, the set $\{a_i^{s^0}\}_{i=1}^k$ is independent of $\{a_i^{u}\}_{i=1}^k$.    
\end{theorem}
\begin{proof}
     With the above notation, $R_i = \begin{bmatrix}
        R_i^u & R_i^{us^0} \\
              & R_i^{s^0}
    \end{bmatrix}$, where $R_i^u \in \mathbb{R}^{n_u\times n_u}$ and $ R_i^{s^0} \in \mathbb{R}^{(m-n_u)\times(m-n_u)}$ are upper triangular. Thus, \cref{eq:eqn_in_R} can be expressed as:
    \begin{equation*}
        \begin{bmatrix}
        R_{i-1}^u & R_{i-1}^{us^0} \\
              & R_{i-1}^{s^0}
    \end{bmatrix} \begin{bmatrix}
    a_i^{u} \\
    a_i^{s^0} 
    \end{bmatrix} = \begin{bmatrix}
    b_{i-1}^{u} \\
    b_{i-1}^{s^0} 
    \end{bmatrix} + \begin{bmatrix}
    a_{i-1}^{u} \\
    a_{i-1}^{s^0} 
    \end{bmatrix},  \quad \text{for } i=1,\dots,K,
    \end{equation*}
    from which we see that the equation for $a_i^{s^0}$,
    \begin{equation}
    \label{eq:ai_s}
         a_{i-1}^{s^0} = R_{i-1}^{s^0}  a_{i}^{s^0} -  b_{i-1}^{s^0}, \quad i=K,\dots,1,
    \end{equation}
    can be solved independently of $a_i^u$ and hence does not depend on $a_i^u$.
\end{proof}

Having computed $a_i^{s^0}$, one obtains $a_i^u$ by solving
\begin{equation}
\label{eq:ai_u}
    R_{i-1}^u a_i^u = b_{i-1}^u + a_{i-1}^u - R_{i-1}^{us^0}a_i^{s^0},\quad i=1,\dots,K.
\end{equation}
For stability, \cref{eq:ai_s} needs to be iterated backward from $i=K$ to $i=1$ since it involves the stable and the neutral subspaces. Moreover, specifying the terminal condition as $a_K^{s^0}=0$ ensures that the condition $f(t_K)^T\psi(t_K) = \bar{J}-J(t_K)$ is satisfied, provided $\gamma_K$ is chosen as specified in \cref{alg:m_equals_nu}. Once $a_i^{s^0}$ is solved for, one can solve \cref{eq:ai_u} forward from $i=1$ to $i=K$ with the initial condition $a_0^u = 0$. \cref{alg:m_>_nu} summarizes the stabilized march algorithm for the case $m>n_u$. 

\begin{algorithm}
\caption{Adjoint shadowing march when $m> n_u$}
\label{alg:m_>_nu}
\begin{algorithmic}[1]
\STATE{Execute until line 13 in \cref{alg:m_equals_nu}, storing $d_i,h_i,Q_i,$ and $R_i$.

\textbf{Compute dimension of the unstable subspace, $n_u$:}
}
\STATE{ Set $j = 1,\;n_u=0$ and compute $\lambda_1 = \frac{1}{T} \sum_{i=1}^K \ln(|[R_{i}]_{11}|)$. }
\WHILE{$\lambda_j > 0$}
\STATE{Set $n_u=n_u+1$.}
\STATE{Set $j=j+1$ and compute $\lambda_j = \frac{1}{T} \sum_{i=1}^K \ln(|[R_{i}]_{jj}|)$.}

\ENDWHILE

\STATE{Using the known values $m$ and $n_u$, identify the splitting $Q_i = [Q_i^u\;\;Q_i^{s^0}]$, $R_i =  \begin{bmatrix}
        R_i^u & R_i^{us^0} \\
              & R_i^{s^0}
    \end{bmatrix} $ and $b_i = [b_i^u,\;\;b_i^{s^0}]^T$ as outlined in \cref{sec:m_greaterthan_nu}.

    \vspace{2mm}
\textbf{Backward march:}    
}

\STATE{Set $a_K^{s^0}=0 \in \mathbb{R}^{m-n_u}$.}
\FOR{$i=K,K-1,\dots,1$}
\STATE{Compute $a_{i-1}^{s^0} = R_{i-1}^{s^0}  a_{i}^{s^0} -  b_{i-1}^{s^0}$.
}
\ENDFOR

\textbf{Forward march:}

\STATE{Set $a_0^u = 0 \in \mathbb{R}^{n_u}$.}
\FOR{$i=1,2,\dots,K$}
\STATE{Solve the triangular system $R_{i-1}^u a_i^u = b_{i-1}^u + a_{i-1}^u - R_{i-1}^{us^0}a_i^{s^0}$.}
\ENDFOR

\textbf{Compute sensitivity:}
\STATE{Concatenate $a_i = [a_i^u\;\;a_i^{s^0}]^T$.}

\STATE{
Compute $\frac{d\bar{J}}{ds} = \frac{1}{T}\left(\sum\limits_{i=1}^K \left(a_i^Td_i + h_i\right) + \int_0^T J_s dt\right)$
}
\RETURN $\frac{d\bar{J}}{ds}$
\end{algorithmic}
\end{algorithm}

\subsection{Cost analysis:}
With $n_u$ known, the stabilized march approach in \cref{alg:m_equals_nu} involves solving $n_u$ homogeneous and $1$ non-homogeneous adjoint equations, along with $1$ primal governing equation. This is similar to the existing NILSS and NILSAS approaches, with the difference that the current approach requires $n_u$ homogeneous adjoint solutions as opposed to $n_u +1$ in NILSAS for computing the neutral subspace. 

On the other hand, each $n_u \times n_u$ triangular solve in the stabilized march approach involves $n_u^2$ operations for backward substitution, along with $n_u$ operations to add the right-hand side of \cref{eq:eqn_in_R}. Therefore, the total cost of solving the linear system with the stabilized march approach is $K(n_u^2+n_u)$, where $K$ is the number of segments for QR decompositions.      

Least squares minimization based approaches, such as NILSS and NILSAS, require solving a KKT linear system. Each matrix-vector product with the KKT matrix involves at least $K$ matrix-vector products with dense matrices of size $n_u \times n_u$, $2(K-1)$ matrix-vector products with upper triangular matrices $R_i$ and $2n_u(K-1)$ additions \cite{nilsas}. Hence, the operation count of a single KKT matrix-vector product is $4Kn_u^2 + 2n_u(K-1)$. If $N_{\text{iter}}$ matrix-vector products are required to solve the KKT system, the total cost would be $N_{\text{iter}}\left(4Kn_u^2 + 2n_u(K-1)\right)$, with $N_{\text{iter}}$ depending on the conditioning of the KKT system or its Schur complement. Since $K\geq 1$ and $n_u \geq 1$, we note that even in the extreme case of $N_{\text{iter}}=1$, the cost of solving the linear system with the stabilized march is significantly less than solving the minimization problem.    

\section{Errors due to discretization}
\label{sec:err_due_to_discretization}
We note that a major difference in the current approach compared to the least squares NILSAS \cite{nilsas} approach was to replace the condition on the neutral adjoint subspace $\frac{1}{T}\int_0^T\psi^Tfdt=0$ by $\psi(T)^Tf(T) = \bar{J}-J(T)$ (\cref{lm:adjoint_particular_time}) and to convert the linear space-time system into a sequence of triangular solves. While the two conditions are equivalent for the continuous-time case, they do not necessarily hold after discretizing the adjoint equations. In the current section, by stating that a term $\tau$ is of order $\mathcal{O}(\Delta t^k)$, we mean $|\tau|\leq C^*\Delta t^k$ for a constant $C^*$. We also take $T$ to be fixed. Typically, for an initial value problem, discretization by a scheme of order $\mathcal{O}(\Delta t^k)$ results in an error in the solution of $\mathcal{O}(\Delta t^k)$ if the initial solution is exact. In the current case, however, the discrete adjoint solution is not equal to the continuous adjoint solution at $t=0$ or $t=T$. 
This is because the unstable, stable and neutral subspaces of the discrete adjoint equation are not the same as the subspaces of the continuous adjoint equation. 
Hence, it is of interest to investigate the error due to discretization. As the next few results show, the error due to discretization in the neutral adjoint subspace is of the same order as that in the stable and the unstable subspaces. 

\begin{lemma}
\label{thm:lyapunov_product_error}
    Let the interval $[0,\;T]$ be discretized using time step of size $\Delta t$, yielding $t_i = i\Delta t,\;i=0,\dots,N$. Assume that the integrations involved in solving the adjoint equations in \cref{alg:m_equals_nu} and \cref{alg:m_>_nu} are of order $\Delta t^k$. Let $\psi_i$ be the discrete solution obtained from the discretized stabilized march and let $\psi^\infty(t)$ be the adjoint shadowing direction of the continuous-time adjoint equation. If $\{\phi_j(u)\}_{j=1}^n$ is the set of Lyapunov covariant vectors of the continuous-time equation with the Lyapunov exponents $\{\lambda_j\}_{j=1}^n$, the error at $t_i$, $\xi_i := \psi_i - \psi^\infty(t_i)$, satisfies
    \begin{equation}
    \label{eq:lyapunov_err}
        \phi_j(t_{i-1})^T\xi_{i-1} = e^{\lambda_j \Delta t}\phi_j(t_{i})^T\xi_{i} + \tau_{i-1}, \quad i=1,\dots,N
    \end{equation}
    where $|\tau_{i-1}| \leq C\Delta t^{k+1}$.
\end{lemma}
\begin{proof}
    Taking $\psi_i$ as the terminal condition at $t=t_i$ and performing exact integration of the adjoint equation, \cref{eq:conventional_adjoint2}, until $t_{i-1}$ yields:
    \begin{equation*}
        \psi(t_{i-1},\psi_i) = e^{-\int_{t_i}^{t_{i-1}} f_u^*d\tau}\psi_i + \int_{t_i}^{t_{i-1}} e^{-\int_{\tau}^{t_{i-1}} f_u^*d\xi}J_u(\tau) d\tau.
    \end{equation*}
    Since $\psi_{i-1}$ and $\psi_i$ solve the discrete adjoint equation of order $\Delta t^k$, the local error, $\tilde{\tau}_{i-1}$, is given by
    \begin{equation*}
        \psi_{i-1} - \psi(t_{i-1}, \psi_i) = \tilde{\tau}_{i-1},
    \end{equation*}
    and is bounded by $\|\tilde{\tau}_{i-1}\|\leq \tilde{C}\Delta t^{k+1}$. Hence, we have
    \begin{equation}
    \label{eq:1ad}
        \psi_{i-1} = e^{-\int_{t_i}^{t_{i-1}} f_u^*d\tau}\psi_i + \int_{t_i}^{t_{i-1}} e^{-\int_{\tau}^{t_{i-1}} f_u^*d\xi}J_u(\tau) d\tau + \tilde{\tau}_{i-1}.
    \end{equation}

    Since the adjoint shadowing direction, $\psi^\infty$, satisfies the continuous-time adjoint equation, we have:
    \begin{equation}
    \label{eq:2ad}
        \psi^\infty(t_{i-1}) = e^{-\int_{t_i}^{t_{i-1}} f_u^*d\tau}\psi^\infty(t_i) + \int_{t_i}^{t_{i-1}} e^{-\int_{\tau}^{t_{i-1}} f_u^*d\xi}J_u(\tau) d\tau
    \end{equation}
    Subtracting \cref{eq:2ad} from \cref{eq:1ad} yields:
    \begin{equation}
    \label{eq:3ad}
        \xi_{i-1} = e^{-\int_{t_i}^{t_{i-1}} f_u^*d\tau} \xi_i + \tilde{\tau}_{i-1}.
    \end{equation}
    Taking the inner product of \cref{eq:3ad} with $\phi_j(t_{i-1})$ and denoting $\tau_{i-1} = \phi_j(t_{i-1})^T\tilde{\tau}_{i-1}$ gives
    \begin{equation}
    \label{eq:4ad}
        \phi_j(t_{i-1})^T\xi_{i-1} = \phi_j(t_{i-1})^Te^{-\int_{t_i}^{t_{i-1}} f_u^*d\tau} \xi_i + \tau_{i-1}.
    \end{equation}
Since $\phi_j(t)$ satisfies $\frac{d\phi_j}{dt} = f_u\phi_j - \lambda_j\phi_j$, we have
\begin{equation}
\label{eq:5ad}
    e^{\int_{t_{i-1}}^{t_{i}} f_u dt} \phi_j(t_{i-1}) = e^{\lambda_j \Delta t} \phi_j(t_i).
\end{equation}
Substituting \cref{eq:5ad} into the first term on the right hand side of \cref{eq:4ad} gives the stated result. Since $\phi_j$ are bounded at all times, $|\tau_{i-1}| = |\phi_j(t_{i-1})^T\tilde{\tau}_{i-1}| \leq C\Delta t^{k+1}$. 
\end{proof}

\begin{remark}
    By rearranging the terms in \cref{eq:lyapunov_err}, the error $\xi_i$ between the discrete adjoint $\psi_i$ and the adjoint shadowing direction $\psi^\infty(t_i)$ satisfies
    \begin{equation}
    \label{eq:7ad}
        \phi_j(t_{i})^T\xi_{i} = e^{-\lambda_j \Delta t}\phi_j(t_{i-1})^T\xi_{i-1} -e^{-\lambda_j \Delta t} \tau_{i-1}, \quad i=1,\dots,N.
    \end{equation}
\end{remark}

\begin{theorem}
\label{cor:err_neutral_subspace}
    Let $\psi_i$ be the discrete adjoint solution outlined in \cref{thm:lyapunov_product_error}. With the terminal condition of the stabilized march approach, $\psi_N^Tf_N = \bar{J}-J(t_N)$, we have
      \begin{equation}
      \label{eq:neutral_subsapce_error}
        \bigg|\frac{1}{T} \int_{i=0}^N f_i^T\psi_i dt\bigg| \leq \left(C_{n_0}^*+CT\right)\Delta t^k, 
    \end{equation}
    where $f_i=f(u_i)$ and $\int_{i=0}^N(\cdot)$ is an $\mathcal{O}(\Delta t^k)$ discretization of the integral $\int_0^T(\cdot)$.
\end{theorem}
\begin{proof}
    The neutral Lyapunov covariant vector, corresponding to $\lambda_{n_0}=0$, is $\phi_{n_0}(t_i) = f_i$. Thus, \cref{thm:lyapunov_product_error} yields:
    \begin{equation*}
        f_{i-1}^T\xi_{i-1} = f_{i}^T\xi_i + \tau_{i-1}. \quad i=1,\dots,N.
    \end{equation*}
    The terminal condition of the stabilized march approach satisfies $f_N^T\xi_N = 0$. Therefore,
    \begin{equation*}
        |f_{N-q}^T\xi_{N-q}| = \bigg|\sum_{p=1}^q \tau_{N-p} \bigg| \leq N C\Delta t^{k+1} \quad 1\leq q \leq N.
    \end{equation*}
    Using the fact that $N=\frac{T}{\Delta t}$, we have $|f_i^T\xi_i| \leq  CT\Delta t^k$ for $0\leq i\leq N$. Hence,
    \begin{equation}
        \label{eq:neutral_err_2}
        \bigg|\frac{1}{T} \int_{i=0}^N f_i^T\xi_i dt \bigg| \leq \frac{1}{T} \int_{i=0}^N |f_i^T\xi_i| dt \leq CT\Delta t^k.
    \end{equation}
    Using the triangular inequality along with \cref{eq:neutral_err_2} yields
    \begin{equation}
    \label{eq:error_triangle_neutral}
      \bigg| \frac{1}{T}\int_{i=1}^N f_i^T\psi_i dt \bigg| -\bigg| \frac{1}{T}\int_{i=1}^N f_i^T\psi^\infty(t_i)dt \bigg| \leq  \bigg|\frac{1}{T} \int_{i=0}^N f_i^T\xi_i dt \bigg| \leq CT\Delta t^k.
    \end{equation}
    Noting that $0=\frac{1}{T}\int_0^T f^T\psi^\infty dt = \frac{1}{T}\int_{i=1}^N f_i^T\psi^\infty(t_i)dt + \mathcal{O}(\Delta t^k)$, we have
    \begin{equation}
    \label{eq:error_neutral_inf}
        \bigg| \frac{1}{T}\int_{i=1}^N f_i^T\psi^\infty(t_i)dt \bigg| \leq C^{*}_{n_0}\Delta t^k
    \end{equation}
    for a constant $C^{*}_{n_0}$. Substituting \cref{eq:error_neutral_inf} into \cref{eq:error_triangle_neutral} yields the stated result.
\end{proof}

\begin{theorem}
\label{thm:error_nonneutral_subspace}
    Let $1\leq j\leq n$ and $j\neq n_0$. If $\psi_i$ is the adjoint solution obtained from the stabilized march approach and $\xi_i$ and $\phi_j(t)$ are as defined in \cref{thm:lyapunov_product_error}, we have the bound:
    \begin{equation}
    \label{eq:ad_bound}
        \bigg|\frac{1}{T} \int_{i=0}^N \phi_j(t_i)^T\xi_i dt  \bigg| \leq \frac{1}{T} \int_{i=0}^N |\phi_j(t_i)^T\xi_i |dt \leq \frac{\tilde{C}}{T |\lambda_j|} + \tilde{C}_j  \Delta t^k 
    \end{equation}
   for $0<\Delta t<\delta$, where $\tilde{C}, \;\tilde{C}_j$ and $\delta>0$ are constants. 
\end{theorem}
\begin{proof}
    In general, $\xi_N \neq 0$ and $\xi_0 \neq 0$ since the discrete stabilized march approach uses the unstable modes of the discrete adjoint equation rather than the continuous adjoint equation. On the other hand, for a given $T$ and $K$, the solutions $\psi_0$ and $\psi_N$ from the stabilized march approach are bounded \cite{stabilized_march_stable,ascher_bvp}. Since $\psi^\infty(t)$ is also bounded, for any Lyapunov covariant vector $\phi_j(t)$, we have a constant $\tilde{C}$ such that $|\phi_j^T(t_0)\xi_0|\leq \tilde{C}$ and $|\phi_j^T(t_N)\xi_N|\leq \tilde{C}$. 

    Let $\phi_j(t)$ be the Lyapunov covariant vector corresponding to $\lambda_j < 0$. From \cref{thm:lyapunov_product_error}, iterating \cref{eq:lyapunov_err} backward from $i=N$ yields:
    \begin{equation}
    \label{eq:9ad}
          \phi_j(t_{N-q})^T\xi_{N-q} = e^{\lambda_j (T-t_{N-q})}\phi_j(t_{N})^T\xi_{N} + \sum_{p=1}^q \tau_{N-p} e^{(q-p)\lambda_j \Delta t}, \quad 1\leq q\leq N.
    \end{equation}
    Since $|\tau_i|\leq C\Delta t^{k+1}$ and $0<e^{\lambda_j\Delta t}<1$ because $\lambda_j<0$, the second term on the right hand side of \cref{eq:9ad} is bounded by 
    \begin{equation*}
        \sum_{p=1}^q \tau_{N-p} e^{(q-p)\lambda_j \Delta t} \leq C\Delta t^{k+1} \left(1+e^{\lambda_j\Delta t}+e^{2\lambda_j\Delta t}+\dots \right) \leq C\Delta t^{k+1} \left(\frac{1}{1-e^{\lambda_j\Delta t}}\right).
    \end{equation*}
    It can be seen that $\lim\limits_{\Delta t\to 0} \frac{\Delta t}{1-e^{\lambda_j\Delta t}} = -\frac{1}{\lambda_j}$. Therefore, $\exists \delta_j>0$ such that for $0<\Delta t<\delta_j$, $\big|\frac{\Delta t}{1-e^{\lambda_j\Delta t}} + \frac{1}{\lambda_j}\big| < 1 \implies \big|\frac{\Delta t}{1-e^{\lambda_j\Delta t}}\big| < 1+\frac{1}{|\lambda_j|}$ . Hence, noting that $|\phi_j(t_N)^T\xi_N|\leq \tilde{C}$, we have the bound for \cref{eq:9ad}:
    \begin{equation}
    \label{eq:8ad}
        |\phi_j(t_i)^T\xi_i| \leq e^{\lambda_j(T-t_i)} \tilde{C} + C\Delta t^k\left(1+\frac{1}{|\lambda_j|} \right), \quad 0<\Delta t<\delta_j,\quad 1\leq i\leq N.
    \end{equation}
    Since $\frac{1}{T|\lambda_j|} \geq \frac{e^{T\lambda_j}-1}{T\lambda_j}=\frac{1}{T}\int_0^T e^{\lambda_j (T-t)}dt = \frac{1}{T}\int_{i=0}^N e^{\lambda_j (T-t_i)}dt + \mathcal{O}(\Delta t^k)$, we also have
    \begin{equation}
\label{eq:discrete_err_1}
        \bigg|\frac{1}{T}\int_{i=0}^N e^{\lambda_j (T-t_i)}dt \bigg| \leq \frac{1}{T|\lambda_j|} + C'\Delta t^k  
    \end{equation}
    for a constant $C'$.
    Using \cref{eq:8ad,eq:discrete_err_1}, we have the bound
    \begin{equation}
    \label{eq:bound_dotproduct}
    \begin{aligned}
         \bigg|\frac{1}{T} \int_{i=0}^N \phi_j(t_i)^T\xi_i dt  \bigg| & \leq \frac{1}{T} \int_{i=0}^N |\phi_j(t_i)^T\xi_i |dt \\
         & \leq  \left(\frac{1}{T}\int_{i=0}^N e^{\lambda_j (T-t)}dt\right)\tilde{C} + C\Delta t^k\left(1+\frac{1}{|\lambda_j|} \right) \\
         & \leq \frac{\tilde{C}}{T |\lambda_j|} + \tilde{C}_j  \Delta t^k  
    \end{aligned}
    \end{equation}
    for $\lambda_j<0$ and $\delta_j>0$. Similarly, for $\lambda_j>0$, one can iterate \cref{eq:7ad} from $i=1$ to obtain $\phi_j(t_i)^T\xi_i = e^{-\lambda_j t_i}\phi_j(t_0)^T\xi_0 + \sum_{p=0}^{i-1}\tau_p e^{-\lambda_j\Delta t(N- p)}$. Since $\xi_0$ is bounded, repeating the same process as above shows that the bound in \cref{eq:bound_dotproduct} also holds for $\lambda_j >0$. Choosing $\delta= \inf\limits_{\substack{1\leq j\leq n\\j\neq n_0}} \{ \delta_j\}$ and noting that $\delta>0$ for finite $n$, we note that the bound in \cref{eq:ad_bound} holds for any $j\neq n_0$.
\end{proof}

\begin{corollary}
    Let $\frac{d\bar{J}}{ds}|_h = \frac{1}{T} \int_{i=0}^N \left(\psi_i^Tf_s(t_i) + J_s(t_i)\right) dt$ be the sensitivity computed using the discrete adjoint solution, $\psi_i$, to the stabilized march approach. Let $\frac{d\bar{J}^\infty}{ds} = \frac{1}{T} \int_{0}^T \left(\psi^{\infty T}(t)f_s(t) + J_s(t)\right) dt$ be the sensitivity from the adjoint shadowing direction. Then, 

    \begin{equation*}
        \bigg|\frac{d\bar{J}}{ds}|_h -  \frac{d\bar{J}^\infty}{ds} \bigg| \leq \frac{C_1}{T} + \left(C_2 + C_3 T\right) \Delta t^k, \quad\quad0<\Delta t < \delta,
    \end{equation*}
    with the constants $C_1$, $C_2$, $C_3$ and $\delta>0$. 
\end{corollary}
\begin{proof}
     Since the covariant Lyapunov vectors form a basis of the phase space, let $f_s(t_i) = \sum_{j=0}^n \tilde{f}_{sj}(t_i)\phi_j(t_i)$, where $\tilde{f}_{sj}(t_i)\in \mathbb{R}$. Let $\sup\limits_{\substack{t\in[0\;T] \\ 1\leq j\leq n}}|\tilde{f}_{sj}(t)|\leq C_{f_s}$, which is bounded since $f_s(u)$ is bounded. Therefore,
    \begin{equation}
    \label{eq:error_sensitivity_inequality}
    \begin{aligned}
         \bigg|\frac{d\bar{J}}{ds}|_h -  \frac{d\bar{J}^\infty}{ds} \bigg| & \leq \bigg|\frac{1}{T}\int_{i=0}^N \xi_i^Tf_s(t_i)dt + \mathcal{O}(\Delta t^k)  \bigg| \\
         &\leq \frac{1}{T} \int_{i=0}^N\bigg| \sum_{j=1}^n \tilde{f}_{sj}(t_i) \phi_j(t_i)^T\xi_i \bigg|dt +  C_J\Delta t^k\\
         & \leq C_{f_s} \sum_{j=1}^n \frac{1}{T} \int_{i=0}^N |\phi_j(t_i)^T\xi_i|dt + C_J\Delta t^k
    \end{aligned}
    \end{equation}
    for a constant $C_J$.
    Using \cref{eq:neutral_err_2} and \cref{thm:error_nonneutral_subspace}, we have
    \begin{equation}
    \label{eq:sum_dot_product_errors}
    \begin{aligned}
        \sum_{j=1}^n \frac{1}{T} \int_{i=0}^N |\phi_j(t_i)^T\xi_i|dt &= \sum_{\substack{j=1 \\ j\neq n_0}}^n \frac{1}{T} \int_{i=0}^N |\phi_j(t_i)^T\xi_i|dt + \frac{1}{T} \int_{i=0}^N |f(t_i)^T\xi_i|dt \\
        & \leq \frac{\tilde{C}}{T} \sum_{\substack{j=1 \\ j\neq n_0}}^n \frac{1}{|\lambda_j|} + \Delta t^k \sum_{\substack{j=1 \\ j\neq n_0}}^n \tilde{C}_j + CT\Delta t^k,\qquad 0<\Delta t < \delta.
    \end{aligned}
    \end{equation}
    Substituting \cref{eq:sum_dot_product_errors} into the final inequality in 
\cref{eq:error_sensitivity_inequality}, we obtain constants $C_1$, $C_2$ and $C_3$ satisfying the inequality stated in the theorem. We note that the constants are bounded because, for a uniformly hyperbolic dynamical system, there is only one neutral subspace \cite{ni_adjoint_arxiv,wang_2013}. Therefore, for a finite-dimensional dynamical system, one has $|\lambda_j|>\epsilon>0$ for $j\neq n_0$ and $1/|\lambda_j| < 1/\epsilon < \infty$.
\end{proof}

\section{Numerical test cases}
The previous sections have discussed the boundary conditions which result in a stabilized march approach, proved that it converges to the correct sensitivity, developed an algorithm of stabilized march for the case $m>n_u$, and investigated that the discretization errors in various adjoint subspaces are of the order of the scheme's local error. Now, we apply the stabilized march to a few test cases of chaotic flows for sensitivity analysis. 

\label{sec:test_cases}
\subsection{Lorenz 63}
The Lorenz 63 system is used to model atmospheric convection \cite{lorentz_63} and is expressed as:
\begin{equation}
    \frac{d}{dt} 
    \begin{bmatrix}
    x \\
    y \\
    z
    \end{bmatrix} =
    \begin{bmatrix}
        \sigma (y-x) \\
        x(\rho - (z-s)) - y \\
        xy - \beta(z-s)
    \end{bmatrix}
\end{equation}
with $s = 0, \sigma=10, \beta=\frac{8}{3}$ and $\rho=25$. With these parameters, the flow is chaotic with a quasi-hyperbolic strange attractor \cite{sparrow_lorenz}, as shown in \cref{fig:Lorenz_attractor}. \Cref{fig:lyapunov_exponents_lorentz} shows its three Lyapunov exponents, computed using \cref{eq:compute_lyapunov_exponents}. The test case has one dimensional unstable, neutral, and stable subspaces. 
\begin{figure}
    \centering
    \includegraphics[scale=0.5]{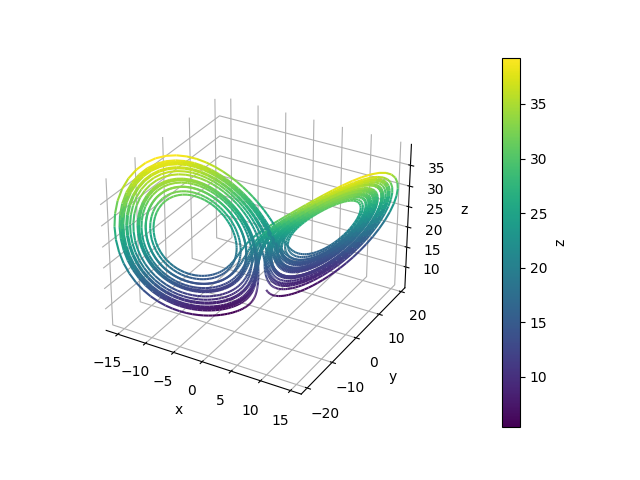}
    \caption{Lorenz attractor}
    \label{fig:Lorenz_attractor}
\end{figure}

\begin{figure}
    \centering
    \includegraphics[scale=0.5]{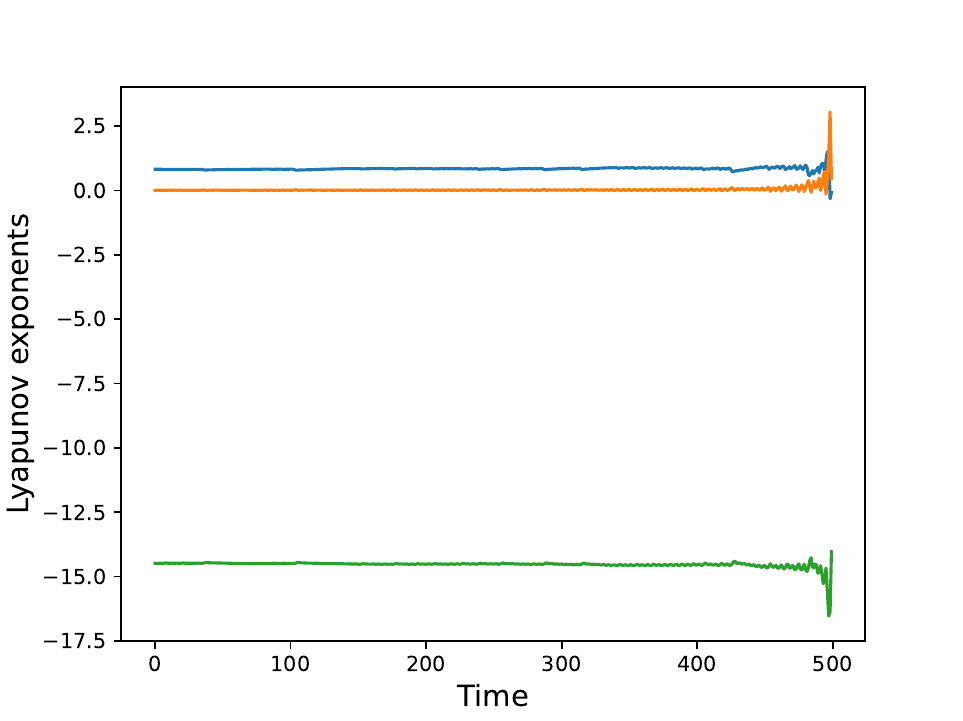}
    \caption{Finite time Lyapunov exponents. }
    \label{fig:lyapunov_exponents_lorentz}
\end{figure}
We are interested in the time average of the state $z(t)$:
\begin{equation}
\label{eq:functional_lorentz}
    \bar{J} = \lim_{T\to\infty} \frac{1}{T} \int_0^T z(t) dt,
\end{equation}
and the parameter of interest is $s$. As noted in \cite{chater_wang_2017}, changing $s$ just translates the entire attractor along the $z$-axis. Thus, the true sensitivity is $\frac{d\bar{J}}{ds} = 1$.

We use the fourth-order accurate explicit Runge-Kutta scheme with a time step of $\Delta t=0.01$ to integrate the primal and the adjoint equations in \cref{alg:m_equals_nu}. The effect of several values of $\Delta T$ on the Lyapunov exponents were investigated, ranging from $\Delta T=0.1$ to $\Delta T=10$. It was found that the Lyapunov exponents were similar until $\Delta T\leq 2.0$, after which the negative Lyapunov exponent was sensitive to and varied with $\Delta T$. In this work, intervals of $\Delta T = 0.2$ were used for the $QR$ decompositions, as has also been used in previous literature \cite{nilsas}.  The integral in the functional, \cref{eq:functional_lorentz}, was discretized using the fourth-order accurate composite Simpson's rule. To verify the rate of convergence, the adjoint stabilized march algorithm (\cref{alg:m_equals_nu}) is run with $m=1$ and spin-up times of $T_{0i}=50$ and $T_{0f}=20$. The plot of the error in sensitivity, $\bigg|\frac{d\bar{J}}{ds} - \frac{d\bar{J}}{ds}\bigg|_{\text{true}} \bigg|$, versus integration length $T$ is shown in \cref{fig:djds_err_vs_T_sqrtT_convergence_lorentz} to verify the $\mathcal{O}\left(\frac{1}{\sqrt{T}}\right)$ rate of convergence (\cref{thm:stabilized_march_2a_T}). The sensitivity at each integration time was calculated by averaging the sensitivities over 20 trajectories obtained with random initial conditions.

\begin{figure}
    \centering
    \includegraphics[scale=0.5]{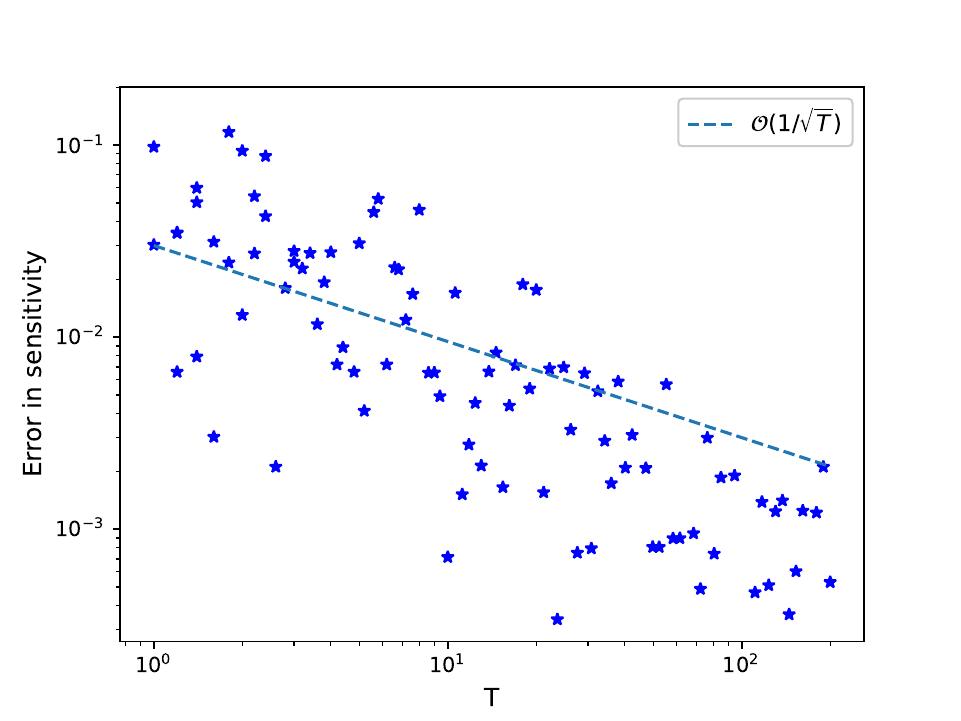}
    \caption{$|\frac{d\bar{J}}{ds} -1|$ vs $T$. }
    \label{fig:djds_err_vs_T_sqrtT_convergence_lorentz}
\end{figure}

\Cref{fig:djds_vs_T_lorentz} shows the mean sensitivities along with the confidence interval of one standard deviation for various parameters with $T$, using 10 random trajectories. It can be seen that the variance in sensitivity decreases in general as $T$ increases.

\begin{figure}
    \centering
    \includegraphics[scale=0.5]{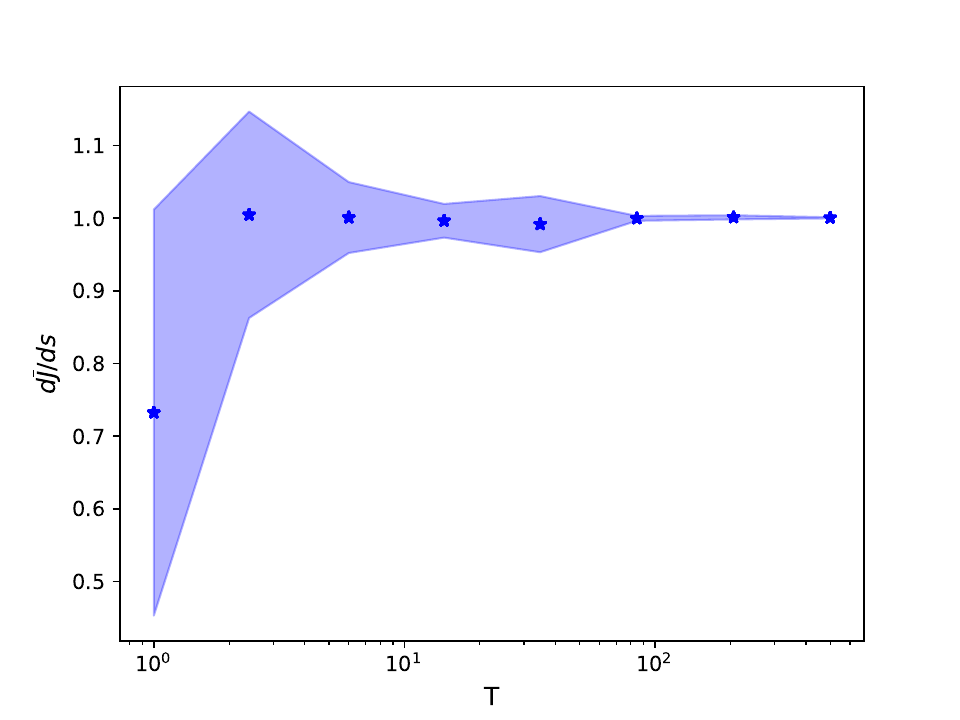}
    \caption{$\frac{d\bar{J}}{ds}$ vs $T$. }
    \label{fig:djds_vs_T_lorentz}
\end{figure}

In order to verify the bound on the average of the neutral adjoint component, \cref{cor:err_neutral_subspace}, we compute $\big|\frac{1}{T}\int_0^T\psi^Tf dt\big|$ for various time steps $\Delta t$, using 10 random trajectories for each $\Delta t$. Since we use RK4, the bound on $\big|\frac{1}{T}\int_0^T\psi^Tf dt\big|$ is expected to converge to zero at $\mathcal{O}(\Delta t^4)$, provided the integrals involved in computing the sensitivity are 4th order accurate. \Cref{fig:adjoint_neutral_convergence_lorentz} shows the convergence of $\big|\frac{1}{T}\int_0^T\psi^Tf dt\big|$ using 10 random trajectories, verifying the derived bound.    
\begin{figure}
    \centering
    \includegraphics[scale=0.5]{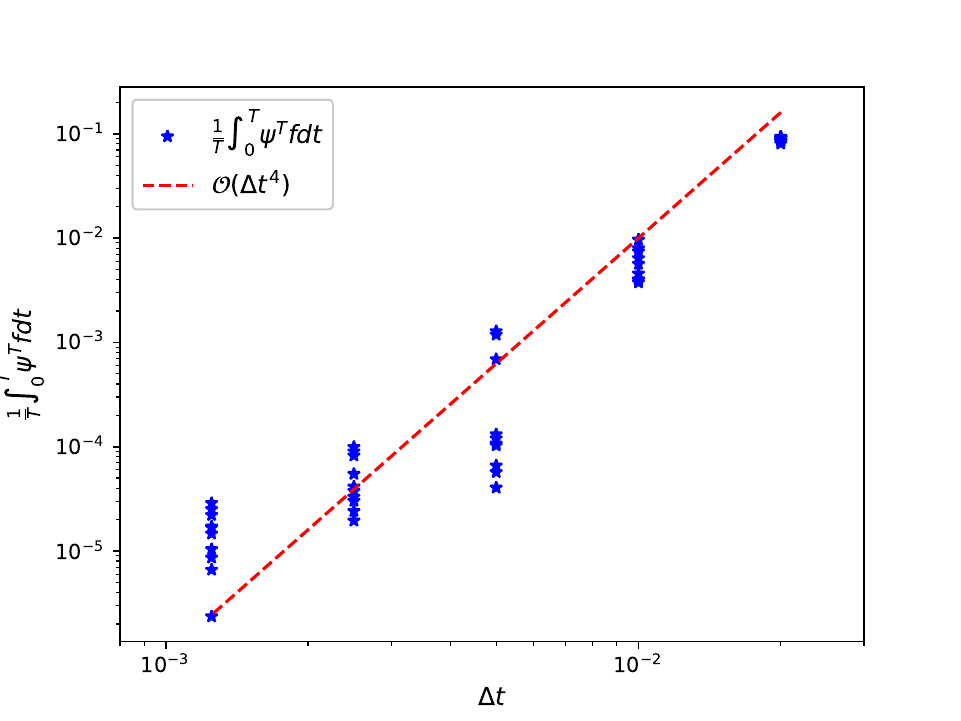}
    \caption{$\big|\frac{1}{T}\int_0^T\psi^Tf dt\big|$ vs $\Delta t$. }
    \label{fig:adjoint_neutral_convergence_lorentz}
\end{figure}

\subsection{Kuramoto-Sivasinsky equation}
The Kuramoto-Sivasinsky (KS) equation is a fourth-order chaotic PDE, given by \cite{blonigan_ks}:
\begin{equation}
\label{eq:ks}
    \frac{\partial u}{\partial t} = -(u+s)\frac{\partial u}{\partial x} - \frac{\partial^2u}{\partial x^2} - \frac{\partial^4u}{\partial x^4} \quad \quad x\in[0\;\;L],
\end{equation}
where $L=128$ and $s=0$. We use the following boundary conditions specified in \cite{blonigan_ks} which make the PDE ergodic:
\begin{equation}
    u(0,t) =  u(L,t) = 0\quad\quad {\frac{\partial u}{\partial x}}_{|{x=0}} = {\frac{\partial u}{\partial x}}_{|{x=L}} = 0 
\end{equation}

The functional is chosen as the time average of the mean solution:
\begin{equation}
    \bar{J}(s) = \lim_{T\to\infty}\frac{1}{T}\int_0^T \left(\frac{1}{L}\int_0^L u(x,t)dx \right)dt
\end{equation}
Spatial discretization is carried out using the finite difference scheme described in \cite{blonigan_ks}, and temporal integration of the primal and adjoint equations is carried out using the third-order explicit Runge–Kutta scheme in \cite{ralston_rk3}, which is known to have a minimum local error bound. The adjoint solution is evolved using the discrete adjoint of the Runge-Kutta method, as outlined in \cref{sec:appendix}.

The following results were obtained using a grid of $\Delta x=1.0$ and $\Delta t = 0.025$. QR decompositions were carried out at intervals of $\Delta T = 5$. Spin-up times were chosen to be $T_{0i}=1000$ and $T_{0f}=50$.  Sensitivities are computed using \cref{alg:m_>_nu} with $m=20$ homogeneous adjoint solutions.  \Cref{fig:primal_ks} shows the plot of the primal solution of the KS equation. \Cref{fig:lyapunov_exponent_ks} shows the plot of the first 20 Lyapunov exponents for the KS equation and indicates that there are around 14 positive Lyapunov exponents.

\begin{figure}
    \centering
    \includegraphics[scale=0.6,trim={5cm 0 0 0},clip]{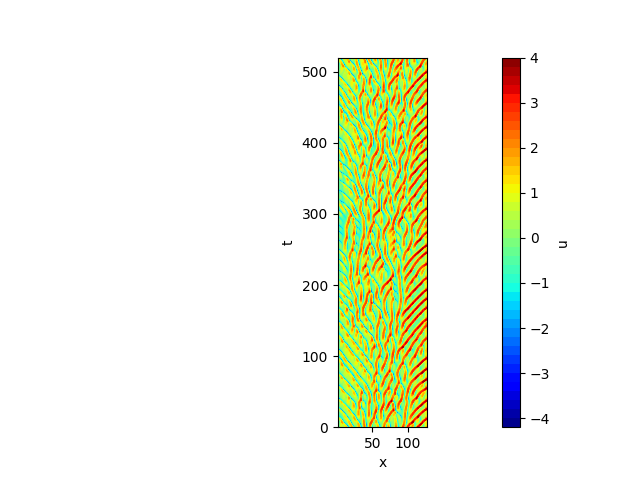}
    \caption{Primal solution $u$ of the KS equation.}
    \label{fig:primal_ks}
\end{figure}

\begin{figure}
    \centering
    \includegraphics[scale=0.5]{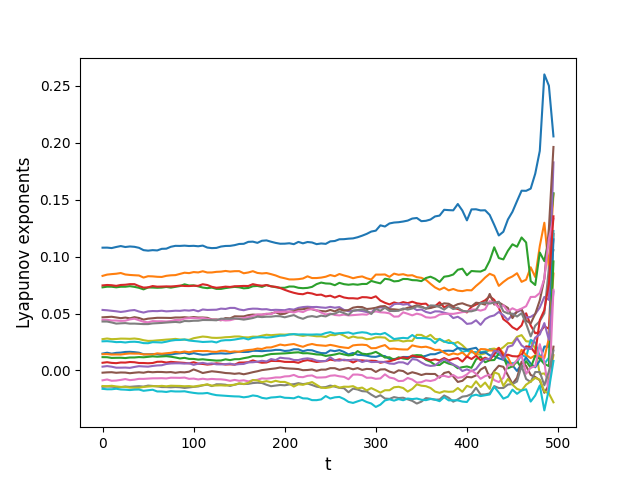}
    \caption{First 20 finite time Lyapunov exponents of KS. }
    \label{fig:lyapunov_exponent_ks}
\end{figure}

\Cref{fig:adjoint_ks} shows a plot of the adjoint solution obtained from \Cref{alg:m_>_nu}. The adjoint solution is bounded at all times. \Cref{fig:djds_vs_T_ks} shows the convergence of sensitivities vs $T$ for $s=0$, with the sensitivity averaged over 10 random trajectories. In previous literature on LSS, it was shown that the sensitivity from LSS approaches the value of $-1$ \cite{blonigan_ks}. In \cref{fig:djds_vs_T_ks}, it is evident that the sensitivities of the stabilized march approach $-1$ as $T$ increases.  

\begin{figure}
    \centering
    \includegraphics[scale=0.6,trim={5cm 0 0 0},clip]{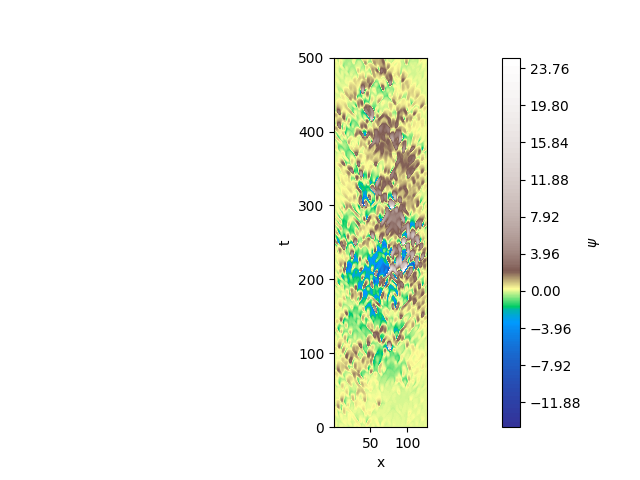}
    \caption{Adjoint solution $\psi$ of the KS equation computed by the stabilized march.}
    \label{fig:adjoint_ks}
\end{figure}

\begin{figure}
    \centering
    \includegraphics[scale=0.5]{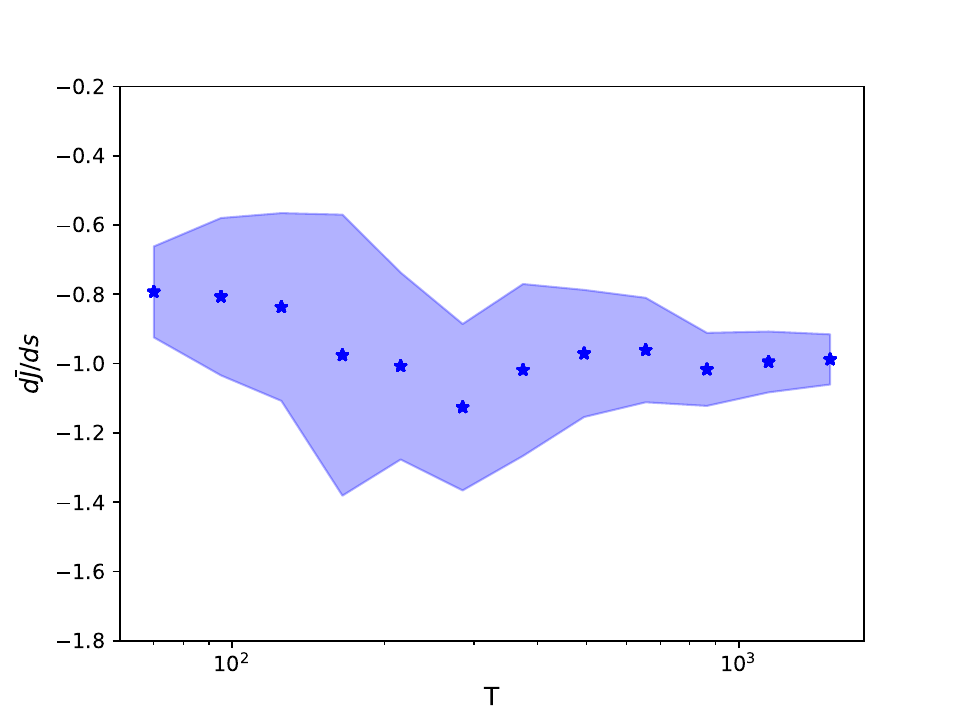}
    \caption{$\frac{d\bar{J}}{ds}$ vs $T$. }
    \label{fig:djds_vs_T_ks}
\end{figure}

In order to verify the bound on the average of the neutral adjoint component, \cref{cor:err_neutral_subspace}, we compute $\big|\frac{1}{T}\int_0^T\psi^Tf dt\big|$ for various time steps $\Delta t$ with $T=100$, using 10 random trajectories for each $\Delta t$. Since we use a third-order accurate RK, the bound on $\big|\frac{1}{T}\int_0^T\psi^Tf dt\big|$ is expected to converge to zero at $\mathcal{O}(\Delta t^3)$, provided the integrals involved in computing the sensitivity are third-order accurate. \Cref{fig:adjoint_neutral_convergence_lorentz} shows the convergence of $\big|\frac{1}{T}\int_0^T\psi^Tf dt\big|$ using 10 random trajectories, verifying the derived bound. 
\begin{figure}
    \centering
    \includegraphics[scale=0.5]{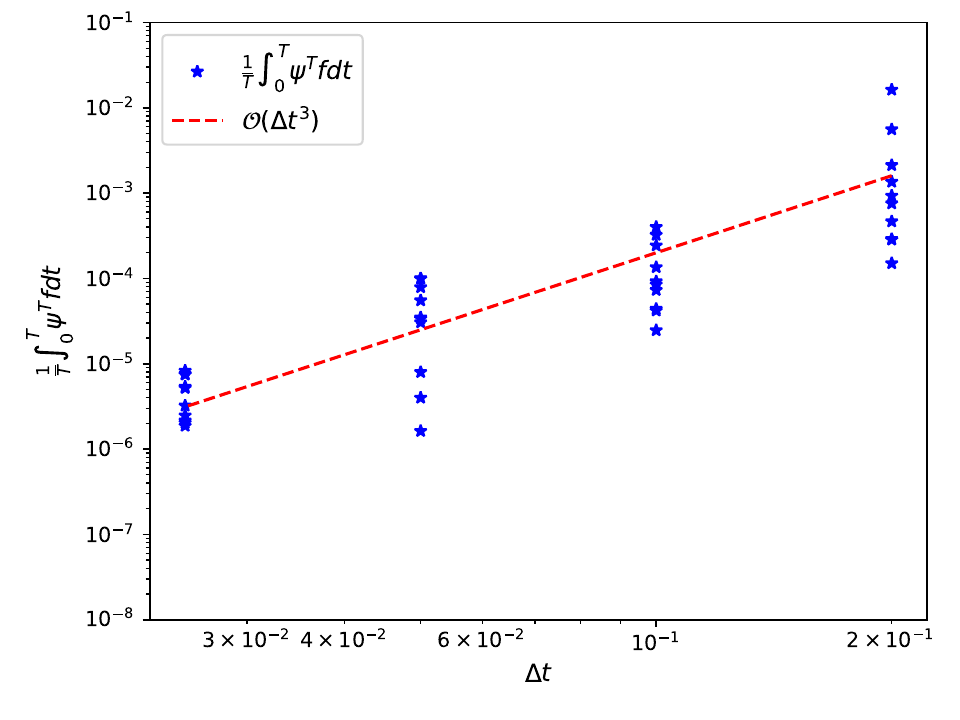}
    \caption{$\big|\frac{1}{T}\int_0^T\psi^Tf dt\big|$ vs $\Delta t$. }
    \label{fig:adjoint_neutral_convergence_ks}
\end{figure}

\section{Conclusion}
The stabilized march approach was developed to compute the adjoint shadowing direction. Specifying a convenient set of boundary conditions converted the problem of solving a space-time linear system into a sequence of triangular solves. It was proven that the method converges to the true sensitivity for large integration times. The approach developed in this work is also applicable when the exact dimension of the unstable subspace is unknown and is overestimated at the beginning of the algorithm. The method was proven to yield sensitivities with errors at the order of the local error of the integration scheme. 

While the method conveniently solves the linear system as a sequence of triangular solves, we note that, similar to NILSAS \cite{nilsas}, we require the unstable subspace which is computed by evolving $m$ homogeneous adjoint solutions backward in time. Compared to NILSAS, the current approach requires one less homogeneous adjoint solution because it does not require the neutral adjoint subspace. However, the cost of a stabilized march would still be high if the number of positive Lyapunov exponents is large. Hence, it would be of interest to investigate possible reduced order modeling strategies to reduce the cost of its implementation. We assumed the shadow trajectory to be a true representation of the perturbed dynamics. However, this might not be the case, as shown in \cite{blonigan_ks, chandramoorthy_wang_nonphysical}. Recent research has investigated the use of a linear response approach to remedy the issue. Since some of the methods based on the linear response approach still require the computation of the shadowing trajectory \cite{ni2023recursive}, the stabilized march for flows developed in this work can be used alongside such methods. 
\appendix
\section{Discrete adjoint of the Runge-Kutta scheme} 
\label{sec:appendix}
Using Runge-Kutta scheme with $n_s$ stages and the butcher tableau $\{a_{ij}\}$, $\{b_i\}$, $i,j=1,...,n_s$, the governing equation $\frac{\partial u}{\partial t} = f(u)$ can be evolved as
\begin{equation}
    \begin{aligned}
      & u_{n+1} = u_n + \Delta t \sum_{i=1}^{n_s} b_if(Y_i) \\
      & Y_i = u_n + \Delta t \sum_{j=1}^{n_s} a_{ij} f(Y_j)\qquad i=1,\cdots,n_s.
    \end{aligned}
\end{equation}
The corresponding adjoint equation, $\frac{d\psi}{dt} + f_u^*\psi+J_u = 0$, is evolved backwards using a discretely consistent scheme of the above RK scheme:
\begin{equation}
\label{eq:rk_adjoint}
    \begin{aligned}
        & \psi_{n} = \psi_{n+1} + \sum_{k=1}^{n_s} \lambda_k \\
        & \lambda_k = \Delta t f_{Y_k}^*\left(b_k\psi_{n+1} + \sum_{j=1}^{n_s} a_{jk}\lambda_j \right) + \Delta t b_k J_{Y_k}
    \end{aligned}
\end{equation}
Similar to the analysis in \cite{sandu_adjoint_rk, hairer_rk}, it can be verified that, for a $p^{\text{th}}$-order RK, the discretization of the adjoint equation in \cref{eq:rk_adjoint} yields an adjoint solution whose first $p$ derivatives at time $t_n$ are the same as that of the continuous adjoint solution for a given $u_n$ and $\psi_{n+1}$, yielding a $p^{\text{th}}$-order accurate scheme.
\newpage

\bibliographystyle{siamplain}
\bibliography{references}
\end{document}